%% file: TinyRho_Aug26_20205_arXiv.tex
\documentclass{article}

 \usepackage[letterpaper,top=2cm,bottom=2cm,left=2.54cm,right=2.54cm,marginparwidth=1.75cm]{geometry}

\pdfoutput=1  

\usepackage[utf8]{inputenc} 
\usepackage[T1]{fontenc}    
\usepackage{url}            
\usepackage{booktabs}       
\usepackage{amsfonts}       
\usepackage{nicefrac}       
\usepackage{microtype}      
\usepackage{xcolor}         

\usepackage{fncylab,enumerate,wrapfig,placeins}

\usepackage{amsmath}

\usepackage{amsthm,amsfonts,amssymb, graphicx}  
\usepackage{relsize} 
\usepackage{accents}  

\usepackage{caption}
\usepackage{textgreek,upgreek,bm}

\usepackage{titlesec}

\newtheorem{theorem}{Theorem}[section]
\newtheorem{lemma}[theorem]{Lemma} 
\newtheorem{proposition}[theorem]{Proposition} 

\newtheorem{corollary}[theorem]{Corollary}

\newlength{\noteWidth}
\long\def\notes#1{\ifinner
	{\footnotesize #1}
	\else 
	\marginpar{\parbox[t]{\noteWidth}{\raggedright\tiny#1}}  
	\fi\typeout{#1}}
	
\def\Lyapsol{\sf M}
\def\pmax{p^{\text{\tiny\sf max}}}

\def\thbias{\upnu}

\usepackage{hyperref} 

\usepackage{cleveref}

\Crefname{corollary}{Corollary}{Corollaries}
\Crefname{eqnarray}{eq.}{eqs.}
\Crefname{equation}{eq.}{eqs.}

\Crefname{figure}{Fig.}{Figs.}
\Crefname{tabular}{Tab.}{Tabs.}
\Crefname{table}{Tab.}{Tabs.}
\Crefname{proposition}{Prop.}{Propositions}
\Crefname{theorem}{Thm.}{Thms.}
\Crefname{definition}{Def.}{Defs.} 
\Crefname{section}{Section}{Sections}
\Crefname{lemma}{Lemma}{Lemmas}
\Crefname{assumption}{Assumption}{Assumptions}

\input{bookmacros_2024.tex}

 \def\bdds#1{\clE^{\text{\rm\tiny\ref{#1}}}}
\def\bdde#1{\varrho^{\text{\rm\tiny\ref{#1}}}}

\makeatletter
\newcommand\gobblepars{%
	\@ifnextchar\par%
	{\expandafter\gobblepars\@gobble}%
	{}}
\makeatother

\def\whamit#1{\smallbreak\pagebreak[3]%
	\noindent\textit{#1}\ \ \gobblepars}

\def\wham#1{\smallbreak\pagebreak[3]%
	\noindent\textbf{#1}\ \ \gobblepars}
	
\def\whamrm#1{\smallbreak\pagebreak[3]%
	\noindent{{\upshape\rm#1}}\ \ \gobblepars}

 \def\tilthetaPRi{\tilde{\theta}^{{\text{\tiny\sf  PR}^i } }}

\title{Revisiting Step-Size Assumptions in Stochastic Approximation}

\author{Caio Kalil Lauand and Sean Meyn
	\thanks{This work was supported by ARO award W911NF2010055,  and NSF grants  
			CCF 2306023,
			EPCN 1935389.} 
	\thanks{Caio Kalil Lauand and Sean Meyn are with the Department of Electrical and Computer Engineering, University of Florida, Gainesville, FL, USA.
		Emails: {\tt\small caio.kalillauand@ufl.edu} and {\tt\small meyn@ece.ufl.edu}}%
}

\begin{document}

\maketitle

\begin{abstract}
	Many machine learning and optimization algorithms are built upon the framework of stochastic approximation (SA), for which the selection of step-size (or learning rate) $\{\alpha_n\}$ is crucial for success.   An essential condition for convergence is the assumption that $\sum_n \alpha_n = \infty$.    Moreover, in all theory to date it is assumed that $\sum_n \alpha_n^2 < \infty$ (the sequence is square summable).   In this paper it is shown for the first time that this assumption is not required for convergence and finer results.  

The main results are restricted to the special case  $\alpha_n = \alpha_0 n^{-\rho}$ with $\rho \in (0,1)$.  The theory allows for parameter dependent Markovian noise as found in many applications of interest to the machine learning and optimization research communities.   Rates of convergence are obtained for the standard algorithm, and for estimates obtained via the averaging technique of Polyak and Ruppert.

\whamb    Parameter estimates converge with probability one, and in $L_p$ for any $p\ge 1$.    
Moreover, the rate of convergence of the the mean-squared error (MSE) is  $O(\alpha_n)$,  which is improved to $O(\max\{ \alpha_n^2,1/n \})$   with averaging.

Finer results are obtained for linear SA:

\whamb   The covariance of the estimates is  optimal in the sense of  prior work of Polyak and Ruppert.  

\whamb   Conditions are identified under which the bias decays faster than $O(1/n)$.  When these conditions are violated,   the bias at iteration $n$ is approximately $\upbeta_\uptheta\alpha_n$ for a vector $\upbeta_\uptheta$ identified in the paper.   Results from numerical experiments illustrate that $\upbeta_\uptheta$ may be large due to a combination of multiplicative noise and Markovian memory.
\end{abstract}

\clearpage

\tableofcontents

\clearpage

\section{Introduction}
\label{s:Intro}

Many problems of interest to the machine learning and optimization research communities hinge upon one task: root-finding in the presence of noise.
That is, the goal is to estimate the vector $\theta^*$ solving $\barf(\theta^*) = 0$, in which  $\barf\colon \Re^d \to \Re^d$ is defined by an expectation
\begin{equation}
	\barf(\theta) \eqdef \Expect[f(\theta,\Phi)]
	\label{e:barfdef}
\end{equation}
where $\Phi$ is a random vector which takes values in a set $\state$,   and $f\colon\Re^d\times\state\to\Re^d$ satisfies appropriate continuity and measurability assumptions.

The standard approach is stochastic approximation (SA), defined as a $d$-dimensional recursion: for an initial condition $\theta_0 \in \Re^d$,  
\begin{equation}
	\theta_{n+1} = \theta_{n}+ \alpha_{n+1} f(\theta_n,\Phi_{n+1}) \, , 
	\qquad 
	n\ge 0 
	\label{e:SA_recur}
\end{equation}
in which $\{\Phi_n\}$ is a sequence of random vectors converging in distribution to $\Phi$ as $n \to \infty$ and $\{\alpha_n\}$ is known as a ``step-size'' sequence.   For ease of exposition we focus on step-size sequences of the form $\alpha_n = \alpha_0  n^{-\rho}$ in which $\alpha_0>0$  and    $\rho \in [0,1]$.

Two canonical examples are
\begin{align}
	\text{Stochastic Gradient Descent (SGD):} &&  
	f(\theta_n,\Phi_{n+1})  & = -\nabla L\, (\theta_n)  +   \Delta_{n+1} 
	\label{e:SGD}
	\\
	\text{Temporal Difference (TD) Learning:} &&  
	f(\theta_n,\Phi_{n+1})  &=    \clD_{n+1} \zeta_n
	\label{e:TD}
\end{align}
In \eqref{e:SGD} the function  $L\colon\Re^d\to\Re$ is a loss function to be minimized, and $\bfDelta$ is a zero-mean sequence (often i.i.d.).
The SA recursion  \eqref{e:TD} describes a temporal difference method, in which $\{ \clD_{n+1} : n\ge 0\}$ is the scalar sequence of temporal differences,  and the sequence of vectors $\{ \zeta_{n} : n\ge 0\}$ are known as eligibility vectors.     In TD-learning with linear function approximation,  the sequence $\{\zeta_n \}$ does not depend upon $\theta_n$ and   $\clB(\theta_n,\Phi_n) \eqdef \Expect[ \clD_{n+1} \mid \Phi_0,\dots,\Phi_n]$ is the Bellman error associated with the parameter $\theta_n$, evaluated at $\Phi_n$  (which summarizes states and actions at iteration $n$).

In these two instances of SA the use of \textit{exploration} implies that  the evolution of  $\bfPhi \eqdef \{\Phi_n\}$ will depend upon the parameter sequence;
one example of this   is the application of $\epsilon$-greedy policies in Q-learning \cite{sutbar18,CSRL}. To address this reality it is assumed that this process satisfies a conditional Markov property as in \cite{pezheu97,rambha19,chedevborkonmey21}.    If the state space $\state$ is finite then   the conditional Markov property requires a  family of transition matrices $\{ P_\theta \colon \theta \in \Re^d \}$,   and for each $n$ and $x'\in\state$ it is assumed that  
\begin{equation}
	\Prob\{\Phi_{n+1} = x' \mid  \Phi_0,\dots,\Phi_n, \theta_n\}    =   P_{\theta_n} (x, x')\,, \qquad \textit{when $\Phi_n=x\in\state$.}
	\label{e:PthetaExplained}
\end{equation}

There is extensive theory for the Markovian setting in which $P_\theta$ does not depend on $\theta$,  and this theory is largely restricted to two  settings: (i) vanishing step-size with $\rho \in (1/2,1]$,
and (ii) constant step-size, in which $\rho=0$. 
In either case, bounds on the mean-squared error (MSE) of estimates are determined by the step-size:   subject to assumptions,  $\Expect[\|  \theta_n - \theta^*\|^2] = O(\alpha_n)$.   These  assumptions are most subtle when $\rho=1$:
see the CLT for SA in \cite{benmetpri12,bor20a} or the MSE theory in \cite{chedevborkonmey21,chezhadoaclamag22}.

The Polyak-Ruppert (PR) averaged estimates are defined by 
\begin{equation}
	\thetaPR_N = \frac{1}{N-N_0+1} \sum_{k=N_0}^N  \theta_k\,,\qquad k\ge N_0
	\label{e:PR}
\end{equation}
in which the interval $\{0,\dots,N_0 \}$ is known as the \textit{burn-in} period.    
Averaging was introduced in prior work to optimize the covariance  in the CLT for the parameter estimates, also known as the  \textit{asymptotic covariance}.     For any SA recursion, the asymptotic covariances associated with estimates with and without averaging
are defined by
\begin{align}
	\SigmaTheta^\alpha &\eqdef   \lim_{n\to\infty} n^\rho \Cov(\theta_n)  
		\label{e:Avar1}
	\\
		\SigmaPR &\eqdef   \lim_{n\to\infty} n  \Cov(\thetaPR_n)  
	\label{e:Avar2}
\end{align}
Moreover, the matrix $\SigmaTheta \eqdef   \lim_{n\to\infty} n \Cov(\theta_n) $, obtained by scaling the covariance of unaveraged estimates by $n$, is typically infinite. Optimality of the asymptotic covariance is defined in a sense similar to the bound of Cram\`er and Rao:   provided $\{\theta_n\}$ is convergent to $\theta^*$ we must have $\SigmaTheta^\alpha  \ge \SigmaPR$,  where the inequality is in the matricial sense.    An expression for $\SigmaPR$ is given in 
\eqref{e:Cov_lowerbound_def}.    Provided the bias decays sufficiently rapidly, 	\eqref{e:Avar2} implies that the MSE satisfies
\begin{equation}
	\lim_{n\to\infty}
	n 	\Expect[\| \thetaPR_n - \theta^* \|^2 ]   = \trace(\SigmaPR)   
	\label{e:PR_MSElimit}
\end{equation}
The error bound $\Expect[\|  \thetaPR_n - \theta^*\|^2] = O(1/n)$ is  far faster than achieved without averaging if $\rho\in [0,1)$.

\textit{However, there is general theory only for the limited range $\rho\in  (1/2,1)$}.  
The major contribution of this paper is to show that these remarkable conclusions from SA theory hold for arbitrary  $\rho\in (0,1)$. 
The results have surprising consequences to step-size selection in practice.   We first recall the motivation for choosing $\rho=0$.

\whamit{Argument for constant step-size.}   
There is intuition based on the noise-free recursion defined by  $x_{n+1} = x_{n}+  \alpha_{n+1}   \barf(x_n) $,  $n\ge 0$,
an Euler approximation of the \textit{mean flow},
\begin{equation}
	\ddt \odestate_t = \barf(\odestate_t)
	\label{e:meanflow}
\end{equation}      
Under the assumptions of this paper, the solutions to \eqref{e:meanflow}  converge to $\theta^*$ with exponential rate of convergence,
and $\{ x_n \}$ also converges geometrically quickly,   provided that  $\rho=0$ and $\alpha_n\equiv \alpha_0$ is chosen sufficiently small.       

It is true that transients  in the parameter error will show a similar decay rate for the corresponding SA recursion. 
This is seen by expressing   \eqref{e:SA_recur} as a \textit{noisy} Euler approximation of the mean flow: 
\begin{equation}
	\theta_{n+1} = \theta_{n}+ \alpha_{n+1} [\barf(\theta_n) + \Delta_{n+1}] \, , 
	\quad 
	n\ge 0 
	\label{e:NoisyEuler}
\end{equation}
with $\Delta_{n+1} = f(\theta_n,\Phi_{n+1}) - \barf(\theta_n)$.   
In proofs of convergence as well as convergence rates for SA, an approximation of the form   $ \theta_n\approx \odestate_{t_n}$ is obtained,   in which $t_n = \sum_{k=1}^n \alpha_k$.    For  fixed step-size the convergence $\odestate_{t_n} \to \theta^*$
holds    geometrically quickly.    For vanishing step-size with $0<\rho<1$ we have  $t_n\approx  \alpha_0  n^{1-\rho}/ (1-\rho)$,  so that  the convergence rate   is not geometric, but it remains    \textit{very fast}.

Of course, the approximation  $ \theta_n\approx \odestate_{t_n}$  is  valid only for a limited range of $n$.
Eventually   noise dominates the rate of convergence:  \textit{The CLT tells us that the rate of convergence of the MSE can be no faster than $O(1/n)$},  which is  far slower than anticipated by the mean flow approximation.   

\whamit{Argument for vanishing step-size.}    
It is argued in this paper that this is a means to attenuate   bias.
This is most clearly explained for the special case of linear SA:   the proof of \Cref{t:stats_summary} is based on the following approximation, for arbitrary  $\rho\in (0,1)$:
\begin{equation}
	\begin{aligned}
		\thetaPR_n & =\theta^*  +     \alpha_n   \upbeta_n   +  \tfrac{1}{\sqrt{n}}     \clZ_n   
		\\[.35em]
		\Expect[\| \thetaPR_n - \theta^* \|^2 ]  &= \alpha_n^2   \| \upbeta_n\|^2  +  \tfrac{1}{n} 	\Expect[\| \clZ_n  \|^2] 
	\end{aligned}
	\label{e:biasVariance}
\end{equation}  
in which $\{\upbeta_n \}$ is a  deterministic sequence,  convergent to an identified value $\upbeta_\uptheta$, and 
$\{\clZ_n \}$ is  a zero-mean stochastic process satisfying $\Cov(\clZ_n)\to \SigmaPR$ as $n\to\infty$.    
Consequently,   if $\rho<1/2$ and $\upbeta_\uptheta\neq 0$,  then the MSE  is dominated by $\alpha_n^2\gg 1/n $.

A compatible bound is also obtained in the present paper for the general   (nonlinear) SA recursion:
\begin{equation}
	\Expect[\| \thetaPR_n - \theta^* \|^2 ]  \leq \bdd{t:BigBounds}  \max( \alpha_n^2 , 1/n) \, , \qquad \rho\in (0, 1)
	\label{e:PR_MSEBound}
\end{equation}
See \Cref{t:BigBounds}.
This bound together with \eqref{e:PR_MSElimit}  were   previously obtained in  \cite{chedevborkonmey21} in the classical regime $ \rho\in (1/2, 1)$.

\wham{\textit{Contributions}:} The main technical contributions are differentiated by the assumptions on the ``noise''  $\bfDelta =\{\Delta_n : n\ge 1\}$.  
\textbf{MD}:   $\bfDelta $  is a martingale difference sequence, 
\textbf{AD} (additive):     $\bfDelta $ does not depend upon the parameter sequence, and   
\textbf{MU} (multiplicative):   the general setting.      

In the MD setting, $\Delta_{n+1}$ may or may not depend upon $\theta_n$.   For example, the MD assumption holds when $\bfPhi$ is i.i.d..   Another example of MD  is  Q-learning  in the tabular setting \cite{sze10,CSRL}, or   using split sampling \cite{bormey00a}.

Results (i)--(iii) are obtained from \Cref{t:BigBounds} in the general setting with $\rho\in (0,1)$:

\whamrm{(i)} \textit{$L_p$ moment bounds}: 
for any $p\ge 1$ 
there exists a constant $\bdd{t:BigBounds}$ depending on $p$,  $\rho$,  and the initial condition such that for each $n\ge 1$,
\begin{equation}
	\Expect[\| \theta_n - \theta^* \|^p ]\leq \bdd{t:BigBounds}  \alpha^{p/2}_n 
	\label{e:Pmoment_SA}
\end{equation}

\whamrm{(ii)}  \textit{Convergence and target bias}:
The   estimates  $\{\theta_n\}$ converge to $\theta^*$ with probability one. Moreover, the rate of convergence of the average \textit{target bias} is   identified: for a constant $\barUpupsilon^* \in \Re^d$,
\begin{equation}
	\lim_{N \to \infty } \frac{1}{\alpha_{N+1} } 
	\Expect\Big[ \frac{1}{N } \sum_{k=1}^N  \barf(\theta_k)\Big]
	=
	\frac{1}{1-\rho}  \barUpupsilon^* \, , \qquad \rho \in (0,2/3)
	\label{e:target_bias_lim}
\end{equation}
A representation for $\barUpupsilon^*$ is given in \Cref{s:LinSA_proofs}, from which we conclude that 
$\barUpupsilon^* = 0$   in the MD and AD settings.

\whamrm{(iii)}  \textit{Convergence rate}: The MSE bound \eqref{e:PR_MSEBound} is established.
\textit{We are not aware of prior work establishing
	almost sure convergence or 
	the MSE  bound \eqref{e:PR_MSEBound}  for $\rho \in (0,1/2]$ even in the MD setting.
}

The sharpest results  are obtained for linear SA,
\begin{equation}
	f(\theta_n,\Phi_{n+1})  =  A_{n+1} \theta_n  -  b_{n+1}   
	\label{e:linearSA}
\end{equation}
in which the matrix $A_n$ and vector $b_n$ are fixed functions of $\Phi_n$ for each $n$.  An example is TD learning with linear function approximation \eqref{e:TD}.

\whamrm{(iv)} 
The approximation \eqref{e:biasVariance} is established, so that  PR averaging achieves the optimal asymptotic covariance for any $\rho\in (0,1)$.  
Moreover, the scaled MSE satisfies 	\eqref{e:PR_MSElimit} provided $\barUpupsilon^* =0$.

\wham{\textit{Takeaway for the practitioner}.}   Up to now we have left open who is the winner:  those who advocate $\rho=0$, or those advocating vanishing step-size with $\rho \in (1/2,1)$.     \textit{There are no winners}.

\whamb If your application falls in the MD or AD settings,    as in typical implementations of stochastic gradient descent,  
then it is sensible to take a small value of $\rho$.    The new theory in this paper demonstrates that in this case 
$\Expect[\| \thetaPR_n - \theta^* \|^2] =O(1/n)$, because  the bias converges rapidly:  $\upbeta_n \to 0$ faster than $O(n^{-D})$ for any $D \geq 1$.   

However,   there is no theory in this paper that suggests $\rho =0$ will give better performance, or adapted non-vanishing step-size such as proposed as part of 
ADAM~\cite{ADAM}.    
Moreover,  constant gain algorithms present the challenge of choosing $\alpha_0>0$ for stability;  no such challenge presents itself for $\rho\in (0,1)$.

\whamb  If your application falls into the class MU then    $\Expect[\| \thetaPR_n - \theta^* \|^2] = O( \| \upbeta_\uptheta\|^2 \alpha^2_n)$.  
The value of $ \| \upbeta_\uptheta\|^2 $ may be large,   as illustrated in numerical experiments.   

Examples of the MU setting abound in reinforcement learning.  In particular,  TD-learning with linear function approximation is an example of linear SA with Markovian noise.

\def\tbullet{{\scalebox{0.5}{\textbullet}}}
\def\sbullet{{\scalebox{0.75}{\textbullet}}}

\wham{\textit{New approach to analysis}:}  

Establishing convergence rates for SA with Markovian noise often begins with  the noise decomposition of M\'{e}tivier and Priouret \cite{metpri87}:
\begin{equation}
	\Delta_{n+1} = \MD_{n+1} - \clT_{n+1} + \clT_{n}  - \alpha_{n+1} \Upupsilon_{n+1}
	\label{e:DeltaDecomp}
\end{equation}
in which $\{\MD_{n+1} \}$ is a martingale difference sequence and $\{\clT_{n+1} - \clT_{n}  \}$ is a telescoping sequence. 

The results of the present paper also rely on a noise decomposition,  introduced here for the first time:   
\begin{equation}
	\Delta_{n+1}= (-\alpha_{n+1})^{m+1} \Upupsilon^{\sbullet}_{n+1} 
	+
	\MD^\sbullet_{n+1}  - \clT^\sbullet_{n+1} +  \clR^\sbullet_{n}
	+ \alpha_{n+1}
	[
	\barG_{n+1}\tiltheta_n
	- \upbeta^\circ_{n+1}]
	\label{e:Deltadecomp_recur_intro}
\end{equation}
where  $\{\MD^\sbullet_{n+1}\} $ is a martingale difference sequence, 
$\{ \Upupsilon^{\sbullet}_{n+1} \}$   is a stochastic process with bounded $L_p$ moments (uniform in $n$ for each $m$), 
and the deterministic sequence of matrices  $ \{\barG_{n+1}\}  $ and vectors $\{ \upbeta^\circ_{n+1} \}$ are convergent. The sequence $\{\clT^\sbullet_{n} - \clR^\sbullet_{n}  \}$ vanishes in $L_p$ at rate $\alpha_{n+1}/n$ so that $\{- \clT^\sbullet_{n+1} +  \clR^\sbullet_{n}\}$ is approximately telescoping. In particular, the sequence $\{ \upbeta^\circ_{n+1} \}$ dominates bias for large $n$.

This decomposition is crucial in establishing optimality of the asymptotic covariance of $\{\thetaPR_n  \}$.  First,  for any given $\rho$ we may choose $m$ so that $(\alpha_{n+1})^{m+1}  \le (\alpha_0)^m /n$, so that the first term is insignificant in covariance calculations.   Second,  a constant such as $\upbeta^\circ_{n+1} $ does not change the covariance.   Finally,   the term  $ \alpha_{n+1} 	\barG_{n+1}\tiltheta_n  $ may be interpreted as a vanishing perturbation of the linear dynamics.  It is shown that such perturbations do not impact the asymptotic covariance---see \Cref{t:stats_summary}.

\subsection{Literature Survey}
\label{s:litreview}
\wham{Asymptotic Statistics}
The optimal asymptotic covariance $\SigmaPR$ was first introduced in the 1950's for the scalar algorithm \cite{chu54o}. The use of averaging to achieve this lower bound appeared much later in \cite{pol90,poljud92,rup88} for general SA recursions with $\rho \in (1/2,1)$ and  $\{\Delta_k\}$ a martingale difference sequence (this is case MD in the present paper).  Under these stronger assumptions on $\bfDelta$, \cite{poljud92} provides a treatment of the regime $\rho \in (0,1)$ for linear SA, obtaining the following conclusions for PR-averaged estimates: optimality of the CLT covariance, optimal MSE rates and almost sure convergence to $\theta^*$.

In applications to optimization it is more common to take $\rho =0$ \cite{bacmou13,vasbacsch19,jinnetjor20,lijunmouwaijor22}. 
The general constant step-size algorithm with averaging is considered in \cite{moujunwaibarjor20,durmounausam24}  for linear SA,  
where it is shown that the estimates $\{\thetaPR_N\}$ are convergent to $\theta^*$ and that the convergence rate is approximately optimal in a mean-square sense.   
Finite-$n$ bounds are also obtained.  It is assumed in \cite{moujunwaibarjor20} that $\bfPhi$ is i.i.d.\  (independent and identically distributed), which implies the MD setting of the present paper.     The paper  \cite{durmounausam24} goes far further,  allowing for $\bfPhi$ to be an uniformly geometrically ergodic Markov chain,   obtaining $L_p$ bounds on the estimation error,  and improving upon the bounds of \cite{moujunwaibarjor20}  in the MD setting.    
We are not aware of extensions beyond the linear setting.

Also with $\rho =0$ are  the articles \cite{laumey25a,huozhachexie24,huochexie22,laumey23a} that construct bias approximations in Markovian settings. 
The assumptions in \cite{laumey23a} are more closely related to the present paper, in which the following bias representation was obtained for linear SA: for a constant vector $\barUpupsilon^* \in \Re^d$ and $\alpha_n \equiv \alpha_0$,
\begin{equation}
	\lim_{N \to \infty}\Expect[\thetaPR_N] = \theta^* + \alpha_0 [A^*]^{-1} \barUpupsilon^* + O(\alpha_0^2)
	\label{e:biasCDC}
\end{equation}
None of this prior work allows $\bfPhi$ to be parameter dependent.   Extensions to parameter dependent noise may be found in   \cite{allgas24}.

Moreover, the CLT covariance lower bound $\SigmaPR$ is generally not achieved in fixed step-size algorithms with averaging even for linear SA in any of the noise settings \cite{moujunwaibarjor20,laumey23a}.
In the general MU case with $\bfPhi$ parameter independent, it admits the following approximation \cite{laumey23a}:   $\lim_{N \to \infty} N\Cov(\tilthetaPR_N) = \SigmaPR +  \alpha_0 Z + O(\alpha_0^2)$ for a matrix $Z \in \Re^{d \times d}$ identified in the paper.

In the theory of reinforcement learning the value $\rho=1$ is often adopted, as in the original formulation of Q-learning by Watkins \cite{sutbar18};  it was discovered in \cite{devmey17b,wai19a} that this choice will result in poor convergence rate unless $\alpha_0$ is chosen sufficiently large.

\wham{Finite-Time Bounds}
The recursion with vanishing step-size $\alpha_n = \alpha_0  n^{-\rho}$ and $\rho \in (0,1)$ has been previously studied within the context of TD learning in \cite{bharussin18} and stochastic gradient descent in \cite{bacmou11}. In both of these papers, finite-$n$ mean squared error (MSE) bounds are obtained in the MD setting for any $\rho$.

In a Markovian setting, analogous MSE bounds have been established for TD learning and general linear SA algorithms with constant step-size \cite{bharussin18,sriyin19}.

The work \cite{chezhadoaclamag22} is most closely related to the present work because they tackled nonlinear recursions with Markovian noise and vanishing stepsize with $\rho \in (0,1)$.   The standard SA recursion \eqref{e:SA_recur} is considered in  \cite{chezhadoaclamag22},  of the form
\begin{equation}
	f(\theta_n,\Phi_{n+1}) = F(\theta_n, X_{n+1})  +  \MD^0_{n+1}
	\label{e:model_SIVA}
\end{equation} 
where $\{X_n\}$ is   a  Markov chain on a state space $\state$ with transition kernel $P$,  and 
$\{\MD^0_n\}$ a martingale difference sequence satisfying the standard assumptions of the SA literature  \cite{bor20a}.  It is assumed that $\{X_n\}$  is \textit{uniformly ergodic} \cite{MT}: there is a unique invariant measure $\pi$ and fixed constants  $R<\infty$ and $\varrho<1$ such that $\|P^n(x,\varble) - \pi(\varble)\|_{\text{\tiny\sf TV}} \le R\varrho^n$.

\begin{subequations}

	In addition, it is assumed that $V(\theta) = \half \|\theta - \theta^*\|^2$ serves as a Lyapunov function for the mean flow, in the sense that  the following holds, for some $c_0>0$ and all $\theta$: 
	\begin{equation}
		\nabla V(\theta) \cdot \barf(\theta) \le -c_0 V(\theta)
		\label{e:Lyap_SIVA}
	\end{equation}    
	Under these assumptions, the mean square error bound is of the form 
	\begin{equation}
		\Expect[  \|   \tiltheta_n \|^2]  \le    K [ \log\big( {n}/{\alpha_0}  \big) +1] \, \alpha_n  +  \epsy_n  \,,\qquad  K =  L  \frac{1}{c_0} \max 
		\Big\{ 1,\frac{\log(R/\varrho)}{\log(1/\varrho)} \Big\} 
		\label{e:chezhadoaclamag22}
	\end{equation}
	in which $\{\epsy_n \}$ vanishes \textit{quickly}, similar to the bound on the right hand side of \eqref{e:bias_add}.   The   constant $L$ takes the form $L = 520 (L_{\bar{f}}+\sigma^2_{\MD^0})^2\alpha_0 (\|\theta^* \| + 1)^2$, where $L_{\bar{f}}$ is  the Lipschitz constant for $\barf$ and $\sigma^2_{\MD^0}$ is a constant  depending upon the   variance of $\MD^0_{k+1}$.  
	
\end{subequations}

The present paper complements  \cite{chezhadoaclamag22},  in that the slow convergence for $\rho<1/2$ is explained by the large bias in this regime when there is multiplicative noise---see \eqref{e:bias_multi}.   Prior to the present paper, it might have been expected that \eqref{e:chezhadoaclamag22} could be improved  using PR-averaging.     The bias formula \eqref{e:bias_multi} also tells us  that the   bound  \eqref{e:chezhadoaclamag22} is loose due to the   $\log(n)$ coefficient, but this may be a necessary price for such an elegant upper bound---see \Cref{fig:SivavsLaumey} and the discussion surrounding it.

Finite-$n$ asymptotic covariance bounds for applications to TD-learning with PR averaging and vanishing step-size appeared recently in \cite{sri24}.

\wham{Organization:} The paper is organized into three additional sections. \Cref{s:main} introduces the assumptions that are imposed throughout the paper, followed by contributions (i)--(iv). \Cref{s:exp} illustrates the theory in \Cref{s:main} through a numerical experiment. Conclusions and directions for future research are included in \Cref{s:conc}. Technical proofs of the main results in \Cref{s:main} are contained on the Appendix.

\section{Main results}
\label{s:main}
\subsection{Assumptions and Notation}
\label{s:assump}
It is assumed that
$\bfPhi \eqdef \{\Phi_n \colon n \geq 0\}$ is  a stochastic process evolving on a Polish state space $\state$. It is parameter dependent, in the sense that its dynamics are governed by a parameterized family of transition kernels $\{ P_\theta \colon \theta \in \Re^d \}$. 
The process $\bfPhi$ need not be Markovian, instead analysis is based upon the Markov chain $\bfPhi^\theta$ with transition kernel $P_\theta$, for $\theta \in \Re^d$ fixed.  

For each $\theta$, $\bfPhi^\theta$  is assumed geometrically ergodic
with unique invariant measure $\uppi_\theta$, so that $\barf(\theta) = \Expect_{\uppi_\theta}[f(\theta,\Phi^\theta_n)]$, where the subscript $\uppi_\theta$ denotes the expectation is taken in steady state: $\Phi^\theta_n \sim \uppi_\theta$ for each $n$. 

Any functions $g: \state \to \Re^d$, $h: \Re^d \times \state \to \Re^d$ are assumed to be measurable with respect to the Borel sigma-algebras $\bx$ and $\clB(\Re^d\times\state)$, respectively.

\wham{Notation:}  

The joint parameter-disturbance process is expressed as $\bfPsi \eqdef \{ \Psi_n = (\theta_n,\Phi_{n+1}) : n\geq 0 \}$.

For any measurable function $w\colon\state\to [1,\infty)$,  let $L^w_\infty$ denote the set of all measurable functions $g\colon\state\to\Re$ satisfying
\begin{equation}
	\| g\|_w \eqdef 
	\sup_{x\in\state} \frac{1}{w(x) } |g(x)| <\infty
	\label{e:Linfty_space}
\end{equation}

\begin{subequations}

	For a  $d$-dimensional vector-valued random variable $X$ and $p\ge 1$,
	the $L_p$ norm is denoted 
	$\| X \|_p = ( \Expect[\| X \|^p])^{1/p}$,  and the $L_p$  \textit{span norm}  $ \|X  \|_{p,s}= \min\{ \| X - c \|_p  : c \in \Re^d \}$.   
	When $p=2$ we have  
	\begin{align}
		\|X  \|_{2,s} & = \sqrt{\trace(\Cov(X))}   
		\label{e:spandef}
		\\
		\|X  \|_{2}^2 & = \|X  \|_{2,s}^2  +  \| \Expect[X] \|^2   
		\label{e:L2andSpan}
	\end{align}
\end{subequations}

Any $d\times d$  positive definite matrix $\Lyapsol$ defines a  norm on $\Re^d$  via $\|x\|^2_{\Lyapsol} = x^\transpose \text{$\Lyapsol$} x$ for $x \in \Re^d$.

\wham{Assumptions:}

The following additional assumptions are imposed throughout the paper:
\wham{(A1)} The SA recursion \eqref{e:SA_recur}
is considered with $\{\alpha_n\}$ of the form $\alpha_n  = \alpha_0 n^{-\rho}$ with $\rho \in (0,1)$ and $\alpha_0>0$.

\wham{(A2)}There exists a function $L\colon\state\to\Re$ satisfying, for all $x\in\state$ and $\theta,\theta'\in \Re^d$,
\[\begin{aligned}
	\| f(\theta,x) - f(\theta',x) \| 
	&\le 
	L(x) \|\theta -\theta'\|
	\\
	\| f(0, x)\| 
	&\leq 
	L(x)   
\end{aligned}
\]

\wham{(A3)}   
The  mean flow ODE \eqref{e:meanflow} is globally asymptotically stable with unique equilibrium $\theta^* \in \Re^d$ and the scaled vector field $\barf_{\infty}(\theta) \eqdef \lim_{c \to \infty } \tfrac{1}{c} \barf(c\theta)$ exists for each $\theta \in \Re^d$. Moreover, the ODE@$\infty$ \cite{bormey00a},
\[
\ddt \odestate^\infty_t = \barf_{\infty}(\odestate^\infty_t)
\]
is globally asymptotically stable.

\wham{(A4)}   
The Markov chain $\bfPhi^\theta$ satisfies (DV3) with common Lyapunov function $V$ and small set $C$:
$$
\left. 
\mbox{\parbox{.85\hsize}{\raggedright
		For  functions $V\colon\state\to\Re_+$,  $ W\colon\state\to [1, \infty)$, 
		a small set $C$, $b>0$,
		\[
		\Expect\bigl[  \exp\bigr(  V(\Phi^\theta_{n+1})      \bigr) \mid \Phi^\theta_n=x \bigr]  
		\le  \exp\bigr(  V(x)  - W(x) +  b \ind_C(x)  \bigl)
		\]
}}
\right\}
\eqno{\hbox{\bf (DV3)}}
$$
for all $x \in \state$, $\theta \in \Re^d$.

In addition, for each $r>0$,
\begin{align}
	S_W(r)  &:= \{ x :  W(x)\le r \} 
	\ \  
	\text{is either small or empty,}
	\\
	 \quad &\sup\{ V(x) :  x\in S_W(r) \}  <\infty \,, 
	\nonumber
	\\
	\lim_{r\to\infty}     &\sup_{x \in \state} \frac{L(x) }{\max \{r,W(x)\}}  = 0 
	\label{e:L_is_oW}
\end{align}

Moreover, for any $p \in (1,\infty)$, the following holds for the family of transition kernels: for a constant $b_d$, any $\theta, \theta' \in \Re^d$ and $H =1+V^p$,
\begin{equation}
	\| P_\theta - P_\theta'      \|_H 
	\leq 
	\frac{b_d}{1 + \|  \theta \| + \|  \theta' \|} \| \theta - \theta'  \|
	\label{e:family_kernel_Lip}
\end{equation}

\wham{(A5)}  $\barf:\Re^d \to \Re^d$ is continuously differentiable in $\theta$, and the Jacobian matrix $\barA = \partial \barf$ is  uniformly bounded and uniformly Lipschitz continuous with Lipschitz constant $L_A$.     
Moreover,  $  A^\ocp \eqdef \barA(\theta^\ocp)$ is Hurwitz.

Assumption (A4) may seem strong at first. The bound (DV3) implies  geometric ergodicity of $\bfPhi^\theta$ for each $\theta \in \Re^d$ \cite{konmey05a} and holds for finite state space Markov chains through state space augmentation under very general assumptions.
Moreover, the Lipschitz condition \eqref{e:family_kernel_Lip} is satisfied by exploration design in applications to reinforcement learning \cite{mey24}.

Assumptions (A2), (A3), and (A5)  imply exponential asymptotic stability of \eqref{e:meanflow}\cite[Prop. A.11]{laumey25a}.
\begin{proposition}
	\label[proposition]{t:meanflowexp}
	Under (A2), (A3), and (A5), the ODE \eqref{e:meanflow} is exponentially asymptotically stable: for positive constants $\bdd{t:meanflowexp}$ and $\bdde{t:meanflowexp}$, and any initial condition $\odestate_0 \in \Re^d$,
	$
	\|  \odestate_t - \theta^*   \| 
	\leq 
	\bdd{t:meanflowexp} \|  \odestate_0 - \theta^*  \| \exp( - \bdde{t:meanflowexp} t) 
	$, $t > 0$.
	\qed
\end{proposition}

\subsection{General Stochastic Approximation}
\label{s:moment}
Moment bounds of the form \eqref{e:Pmoment_SA} for general nonlinear SA recursions \eqref{e:SA_recur} with Markovian noise are established  in \cite{chedevborkonmey21} for $p=4$ under slightly weaker assumptions as in the present paper. In particular, the authors consider a relaxation of the condition \eqref{e:L_is_oW} in (A4) (see \Cref{s:moment} for more details). However, this prior work concerns the $\rho \in (1/2,1)$ regime for vanishing step-size algorithms. 
An extension to constant step-size is presented in \cite{laumey23a} where $L_4$ moment bounds on estimation error are also established in a parameter independent setting. 

Similar arguments as the ones used in these prior works can be used to establish moment bounds in the setting of this paper. Under (A4), we obtain bounds of the form \eqref{e:Pmoment_SA} for arbitrary $p$, which in turn imply contributions (ii) and (iii).
\begin{theorem}
	\label[theorem]{t:BigBounds}
	Suppose (A1)--(A4) hold. Then, for each $\rho \in (0,1)$,
	\begin{romannum}

		\item There exists $\bdd{t:BigBounds} < \infty$ depending upon $\Psi_0$ and $\rho$ such that \eqref{e:Pmoment_SA} holds. 

		\item The sequence of estimates converges to $\theta^*$ for each initial condition $\theta_0 \in \Re^d$ with probability one: $ \theta_n \overset{\as}{\to} \theta^\ocp$.
		
		If in addition (A5) holds,
		
		\item There exists $\bdd{t:BigBounds}$ depending upon $\Psi_0$ and $\rho$ such that \eqref{e:PR_MSEBound} holds.
		
		\item The representation for the average target bias in \eqref{e:target_bias_lim} holds.
		Moreover, the right hand side of  \eqref{e:target_bias_lim}  is zero outside of the MU setting.

	\end{romannum}
\end{theorem}

The proofs of part (i) and (ii) of \Cref{t:BigBounds} are given at the end of \Cref{s:momentproofs}, while parts (iii) and (iv) are proven in \Cref{s:GenSA_proofs}.

\Cref{t:BigBounds}~(iv) hints 
at 
a deeper connection between performance of slowly vanishing and constant step-size algorithms. See \Cref{t:stats_summary} for a bias representation entirely analogous to \eqref{e:biasCDC} for linear SA, which is obtained as a consequence of \eqref{e:target_bias_lim}.

\subsection{Linear Stochastic Approximation}
\label{s:LinSA}
This section concerns asymptotic statistics for linear SA recursions of the form \eqref{e:linearSA}. In the general MU setting, it is assumed that $A_{n+1} = \Expect_\uppi[A(\Phi_{n+1})] = A^*$ and $b_{n+1} = \Expect_\uppi[b(\Phi_{n+1})] = \barb$. Note that $A_n = A^*$ for each $n$ when the noise is AD. The mean vector field takes the form $\barf(\theta) = A^*(\theta - \theta^*)$ with $\theta^* = [A^*]^{-1} \barb$.

The optimal asymptotic covariance $\SigmaPR$ is defined as follows:
\begin{equation}
	\SigmaPR  \eqdef  G \Sigma_{\MD^*} G^\transpose
	\label{e:Cov_lowerbound_def}
\end{equation}
where $  G= [A^*]^{-1}$  
and $ \Sigma_{\MD^*} $ is the asymptotic covariance matrix of the martingale difference sequence $\{\MD^*_n\}$, depending only on the ``noise'' $\bfPhi$.
See \eqref{e:parameter_indep} in \Cref{s:GenSA_proofs} for a precise definition of $\clW^*$ in terms of algorithms primitives.

\begin{subequations}

	\begin{theorem}
		\label[theorem]{t:stats_summary}
		Suppose (A1)--(A5) hold for SA recursions of the form \eqref{e:linearSA}. Then, for each $\rho \in (0,1)$,
		\begin{romannum}
			
			\item 	If the noise is AD, there exist $\bdd{t:stats_summary}$, $\bdde{t:stats_summary}$ and $n_b>0$ sufficiently large such that
			\begin{equation}
				\|	\Expect[\tiltheta_n] \|
				\leq
				\bdd{t:stats_summary} \exp( - \bdde{t:stats_summary} (\tau^b_n - \tau^b_{n_b}))
				\label{e:bias_add}
			\end{equation}
			in which $  \tau^b_n =\alpha_0 (1+(1-\rho)^{-1}[n^{1-\rho} - 1])$.
			
			\item	If the noise is MU and $\rho \in (0,2/3)$, there exists a vector $\upbeta_\uptheta \in \Re^d$ such that
			\begin{equation}
				\lim_{n \to \infty} \frac{1}{\alpha_{n+1}} \Expect[\tilthetaPR_n]  =   \upbeta_\uptheta 
				\, , \qquad \tilthetaPR \eqdef \thetaPR - \theta^*
				\label{e:bias_multi}
			\end{equation}

			\item For any of the three noise settings (including MU),
				$
				\lim_{n \to \infty} n\Cov(\thetaPR_n) = \SigmaPR
				$
			
		\end{romannum}
	\end{theorem}
	The asymptotic bias in \eqref{e:bias_multi}  may be expressed   $  \upbeta_\uptheta = (1-\rho)^{-1} [A^*]^{-1} \barUpupsilon^* $,  in which a  representation for $\barUpupsilon^*$ is given in \Cref{s:LinSA_proofs}.
	
\end{subequations}

\begin{subequations}

	The limits in 	\Cref{e:bias_add,e:bias_multi}
	come with error bounds:

	\begin{theorem}
		\label[theorem]{t:stats_summary_better}
		Under the assumptions of \Cref{t:stats_summary} we have the following representations: 
		for a constant $ \bdd{t:stats_summary_better}$   depending upon $\Psi_0$ and $\rho$,  
		\begin{romannum}
			\item
			\begin{equation}
				\Expect[\tilthetaPR_{n}]  =  \alpha_{n+1} \upbeta_\uptheta + \epsy_n^\upbeta 
				\, , 
				\qquad 
				\| 
				\epsy_n^\upbeta  \|    \le   \begin{cases}
					\bdd{t:stats_summary_better} n^{-3\rho/2} \quad &  \rho\le 1/2
					\\
					\bdd{t:stats_summary_better}	n^{\rho/2-1} \quad &  \rho >1/2
				\end{cases}
				\label{e:biasExp}
			\end{equation}

			\item  	Denoting $\sigmaPR \eqdef \sqrt{ \trace(\SigmaPR) }$, 
			\begin{equation}
				\begin{aligned}
					\| \tilthetaPR_n \|_2   \leq     \alpha_{n+1} \|  \upbeta_\uptheta  \| + \tfrac{1}{\sqrt{n} } & \sigmaPR    +  \epsy_n^\upsigma
					\,, \qquad   
					|\epsy_n^\upsigma |      \le   \begin{cases}
						\bdd{t:stats_summary_better}  n^{-3\rho/2}   \quad &  \rho\le  3/7
						\\
						\bdd{t:stats_summary_better}  	n^{-(3-\rho)/4}  	    \quad &  \rho  > 3/7  
					\end{cases}			 
				\end{aligned}
				\label{e:SecondMomentBddLinearSA}
			\end{equation}
		\end{romannum}
		Consequently,  if $\|  \upbeta_\uptheta  \| $ is non-zero, then the MSE $ \| \tilthetaPR_n \|_2^2$ converges to zero at the optimal rate $1/n$ if and only if $\rho>1/2$. 
	\end{theorem}
	
\end{subequations}

The proof of parts (i) and (ii) of \Cref{t:stats_summary} are given at the end of \Cref{s:Bias_proofs}. A proof of (iii) for the AD setting can be found at the start of \Cref{s:Cov_proofs}, while the proof for the MU setting is given at the end of \Cref{s:Cov_proofs}.
\Cref{t:stats_summary_better} is obtained by identifying the dominant error terms in  \Cref{t:stats_summary}.

\section{Numerical Experiments}
\label{s:exp}

The numerical experiments contained in this section aim to illustrate the general theory, focusing on a simple linear model for which the asymptotic bias and covariance are easily computed.

An instance of the recursion \eqref{e:linearSA}  is considered in which $\{\Phi_n \}$ evolves on $\state = \{ 0, 1 \}$, with transition matrix 
$P = \begin{bmatrix}
	a & 1-a
	\\
	1-a & a
\end{bmatrix}$ where $a \in (0,1)$.    The Markov chain is reversible, with  uniform invariant distribution $\uppi$.   Moreover, it satisfies the assumptions of \cite{chezhadoaclamag22}: 
it is uniformly ergodic and its rate of convergence can be identified: $\|P^n(x,\varble) - \pi(\varble)\|_{\text{\tiny\sf TV}} \le R\varrho^n$ with $R =1/2 $ and $\varrho = 2a-1$.

For a pair of matrices and a pair of vectors  $A^0,A^1 \in \Re^{d \times d}$ and $b^0,b^1 \in \Re^{d }$,  we have \eqref{e:linearSA} with
\begin{subequations}
	\begin{equation}
		A_{n+1} = \Phi_n A^1 + (1-\Phi_n)A^0  
		\qquad
		b_{n+1} = \Phi_n b^1 + (1-\Phi_n)b^0 
		\label{e:Affine_spectral}
	\end{equation}
	Since $\uppi$ is uniform, the mean flow vector field is
	\begin{equation}
		\barf(\theta) = A^*\theta - \barb \quad 
		\text{with } A^* = \tfrac{1}{2}(A^0 + A^1) \, , 
		\quad \barb= \tfrac{1}{2}(b^0 + b^1)
		\label{e:barf_linexp}
	\end{equation}
	\label{e:Affine_model_exp}
\end{subequations}

Closed-form expressions for  the terms in \Cref{t:stats_summary} are summarized in the following.  

\begin{lemma}
	\label[lemma]{t:LinspecEx}
	For    $0<a<1$ the asymptotic statistics of the  linear   model  are computable, with
	\begin{romannum}
		\item Asymptotic bias \eqref{e:bias_multi}:    
		\begin{align}
			\upbeta_\uptheta =   \frac{1}{1-\rho} [A^*]^{-1} \barUpupsilon^*   \quad \textit{with} \quad    
			\barUpupsilon^* = 
			\frac{ (2a-1)}{4(1-a)}(A^1 - A^0)(A^0 \theta^* - b^0)
			\label{e:UpupsilonEx}
		\end{align}
		
		\item   Optimal asymptotic covariance \eqref{e:Cov_lowerbound_def}:
		$\displaystyle		\SigmaPR =  \frac{a}{1-a}(A^0 \theta^*) (A^0 \theta^*)^\transpose
		$.
	\end{romannum}
\end{lemma}
\begin{wrapfigure}[17]{r}{0.5\textwidth}	
	\includegraphics[width= 1\hsize]{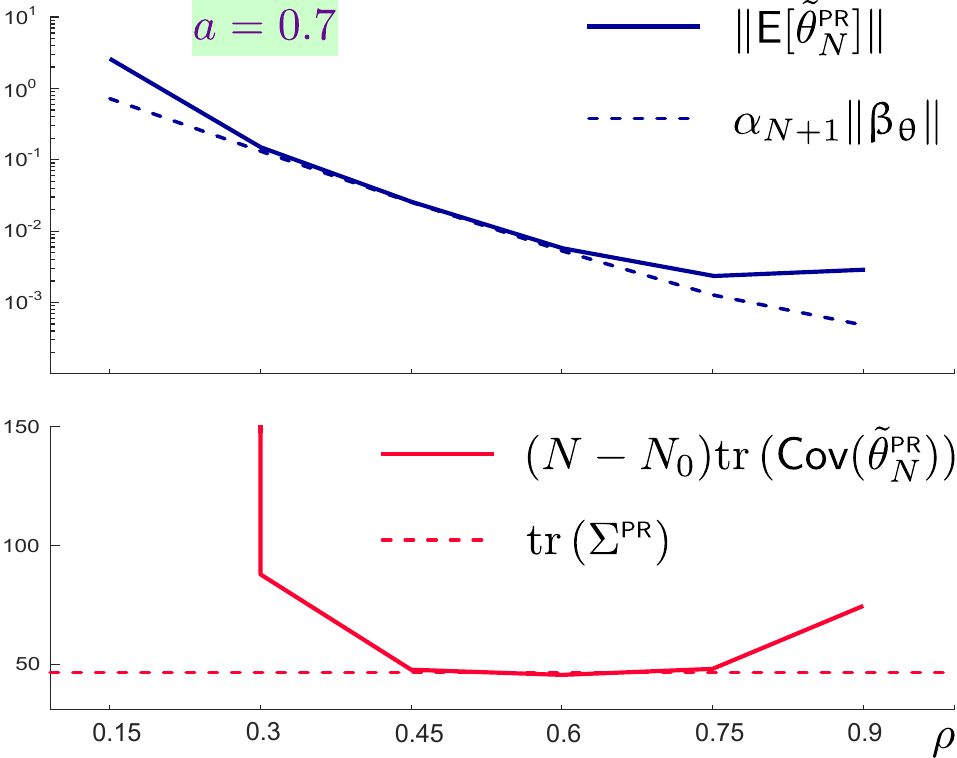}  
	\caption{Empirical and theoretical mean and variance.} 
\label{fig:BiasVar}
\end{wrapfigure}
The proof of \Cref{t:LinspecEx} is postponed to \Cref{s:exp_proofs}.


\wham{Setup:} 
In  the experiments surveyed here,
the linear SA recursion \eqref{e:linearSA}  using \eqref{e:Affine_spectral} was simulated with $a \in \{0.3,0.5,0.7\}$ and
\begin{equation}\begin{aligned}
	A^0 &= 2 \begin{bmatrix}
		-2 & 0
		\\
		1 & -2
	\end{bmatrix}
	\, , \quad 
	b^0 =  \begin{bmatrix}
		0
		\\
		0
	\end{bmatrix}
	\\
	A^1 &= 2 \begin{bmatrix}
		1 & 0
		\\
		-1 & 1
	\end{bmatrix}
	\, , \quad 
	b^1 =  -2\begin{bmatrix}
		1
		\\
		1
	\end{bmatrix}
\end{aligned}
\label{e:expchoices}
\end{equation}
so that $\barf(\theta) = A^* \theta -\tfrac{1}{2}b^1$ with $A^*= - I$ and $\theta^* = - \tfrac{1}{2}b^1$.

Several choices of step-size were considered, in which $\alpha_n = \alpha_0 n^{-\rho}$ with  $\alpha_0 =0.5$ and  $\rho \in \{0.15,0.3,0.45,0.6,0.75,0.9\}$. For each value of $\rho$, $M=300$ independent experiments were carried out for a time horizon of $N=3 \times 10^{5}$ with initial conditions $\{\theta_0^i : 1 \leq i \leq M\}$ sampled independently from $N(\theta^*,25I)$.

After obtaining the sequence of estimates $\{\theta_n^i\}$ for the $i^{\text{th}}$ experiment, PR-averaging was applied with $N_0 = 2\times 10^3$ to compute $\{{\tilthetaPRi_{n}} : n> N_0\}$.


\wham{Results:} The empirical mean and variance were calculated for the final PR-averaged estimates ${\thetaPR_{N}}^i$ across all independent runs  and are plotted as functions of $\rho$ in 
\Cref{fig:BiasVar} for the special case $a=0.7$. Also plotted in this figure are their associated theoretical optimal values obtained from \Cref{t:stats_summary}.  Results for the remaining choices of $a$ are displayed in \Cref{fig:BiasVar03,fig:BiasVar05} in \Cref{s:exp_proofs}. We see that the empirical mean and covariance are near optimal for $\rho \in \{0.45,0.6,0.75\}$.

\whamit{Small-$\rho$ regime.}
Problems are observed for small $\rho$ as expected in view of
\eqref{e:SecondMomentBddLinearSA} since $\upbeta_\uptheta\neq 0$  (see \eqref{e:UpupsilonEx}).    Moreover, for small $\rho>0$ there is a need for a longer run time because the step-size vanishes so slowly.  For example,   when $\rho = 0.15$ we have that $\alpha_N \approx 0.09$ and if the same experiments were carried out with $\alpha_n \equiv  \alpha_N$ for all $n$, one could not expect optimality of the asymptotic covariance based upon theory in \cite{laumey23a,moujunwaibarjor20}.

\whamit{Large-$\rho$ regime.}  
Also observed in \Cref{fig:BiasVar} is poor solidarity with theory for $\rho\sim 1$.    This is predicted by \Cref{t:stats_summary_better}  since the error terms 
$ \epsy_n^\upbeta $, $ \epsy_n^\upsigma$ converge to zero at rate $	n^{-(3-\rho)/4}  \approx 	n^{-1/2}$ when $\rho\approx 1$.

\textit{Performance without averaging.} 
Shown on the left hand side of \Cref{fig:SivavsLaumey} is a plot of the approximation for the MSE $\{ \| \tiltheta_n \|_2^2  \}$,  along with the  finite-$n$ bounds \eqref{e:chezhadoaclamag22}. See \Cref{s:exp_proofs} for details on how each of the terms in \eqref{e:chezhadoaclamag22} were identified for this example.
The bound is very loose even though the maximum in  \eqref{e:chezhadoaclamag22} defining $K$  is equal to unity.  However, remember that the bound of 
\cite{chezhadoaclamag22} is universal, over all nonlinear SA algorithms with common values of  $K$.   

\textit{Averaging.} 
The plots on the right hand side of \Cref{fig:SivavsLaumey} show an approximation for the MSE
$\{ \| \tilthetaPR_n \|_2^2\}  $, along with the  approximate MSE obtained from \Cref{t:stats_summary_better}    (obtained by dropping 	the error term $ \epsy_n^\upsigma$ in \eqref{e:SecondMomentBddLinearSA}).  The approximation is surprisingly tight over the entire run.

\begin{figure}[h]
\centering
\includegraphics[width=\hsize]{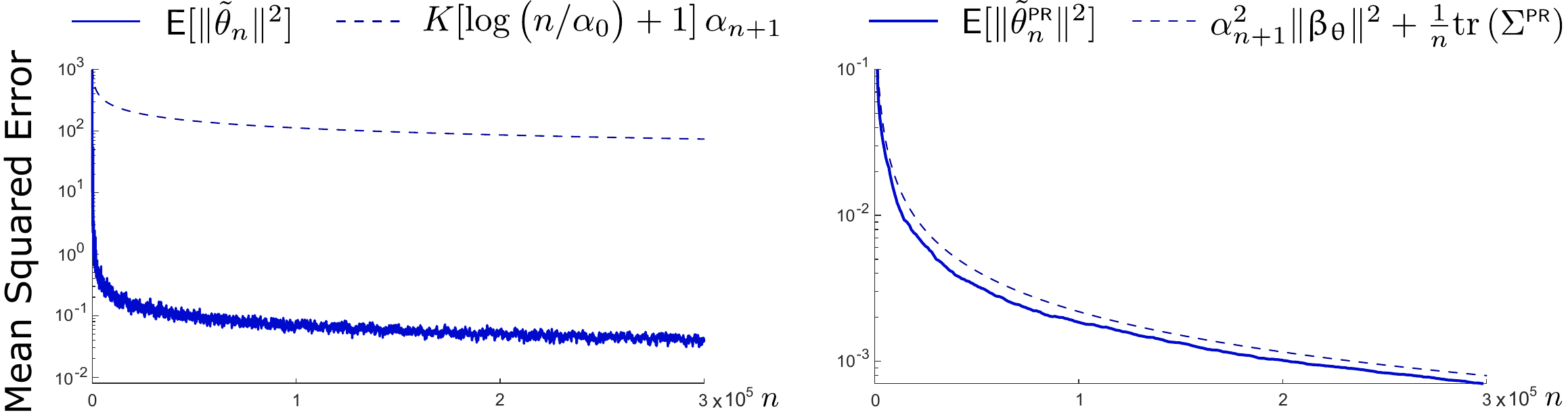}
\caption{Comparison between empirical and theoretical mean squared error with and without PR averaging.}
\label{fig:SivavsLaumey}
\end{figure}

\section{Conclusions}
\label{s:conc}

The condition $\rho>1/2$ is imposed in classical treatments of stochastic approximation because of the standard requirements,
\[
\sum_{n=1}^\infty \alpha_n = \infty \,, \qquad 
\sum_{n=1}^\infty \alpha_n^2  < \infty \,,  
\]
in which the first condition is essential for sufficient exploration.   Since SA was introduced in the 1950s, the second condition has been deemed essential to obtain convergence of the recursion.  This paper establishes convergence for the entire range of $\rho\in (0,1)$ based on an entirely new proof technique,   in which the first step is establishing  convergence rates in $L_p$ for any $p$.    
This is also the first paper to quantify the impact of \textit{bias} when $\rho<1$.

It is worth repeating that the implications for the practitioner depend on the setting:

\wham{$\circ$}  In the MD or AD settings, any value of $\rho\in (0,1)$ will achieve the optimal rate of convergence of the MSE  (that is, $O(1/n)$) for the Polyak-Ruppert averaged estimates.  

\wham{$\circ$} If   the noise is multiplicative (MU setting), then   $\rho>1/2$ is required to achieve the optimal rate of convergence after the application of averaging.

There are several paths for additional research:

\whamb
It is known that
$\lim_{n\to\infty }  n \Expect[\| \thetaPR_n - \theta^* \|^2 ]  =\trace(\SigmaPR)$ for $\rho\in (1/2,1)$ for nonlinear SA, even if  $\barUpupsilon^* \neq 0$  \cite{chedevborkonmey21}.    Can this be extended to  $\rho\in (0,1/2]$  when  $\barUpupsilon^* =0$?     Can  a version of \eqref{e:biasVariance}   be established outside of the linear setting when  $\barUpupsilon^* \neq 0$?

\whamb  In applications to reinforcement learning and gradient-free optimization, the ``noise''  $\bfPhi$ is designed by the user for the purpose of ``exploration''. Recent success stories show that it is possible to select $\bfPhi$ so that $\trace(\SigmaPR)$ and  $\barUpupsilon^*$ are minimized for a deterministic version of SA known as quasi-stochastic approximation \cite{laumey22e,laumey22d}, but can these ideas be extended to the stochastic setting?

\whamb  Theory in non-convergent settings, such as in non-convex optimization,  will require new performance metrics.

\whamb  Although this paper focused on general single timescale SA algorithms, several algorithms of interest to the machine learning and optimization research communities are built upon the framework of two time-scale SA. One example being actor-critic algorithms \cite{kontsi00,kon02,sutbar98}. SA in two timescales with Markovian noise has been studied before for several special cases \cite{karbha18,kalmounautadwai20a,faibor23,doa21}, but can we extend these results to the setting of the present paper?

\clearpage

\bibliographystyle{abbrv}
\bibliography{strings,markov_nolink,q,QSA,ESC,bandits,extras,extrarefs_caio} 


\appendix

\clearpage

\centerline{\Large \bf Appendix}

\bigskip

The following contains proofs of the main results and further details on numerical experiments.

\section{Markovian Background}
\label{s:Markov_sec}
In this section only, we consider a single aperiodic, $\psi$-irreducible Markov chain $\bfPhi$ satisfying (A4) with unique invariant measure $\uppi$.



We often write  $\uppi(g)$ or $\barg$ for the mean  $ \int g(x) \uppi(x)$.  The function $g$ may be real- or vector-valued.
Recall the definition of $\| g\|_w$ for measurable functions $g,w$ on $\state$ in \eqref{e:Linfty_space}. If $\uppi(w)$ is finite then so is $\uppi(|g|)$, and in this case we denote
\[
\tilg(x) = g(x) - \barg \,,   \qquad  x\in\state  \, .
\]

The function  $\hag$ is said to be the solution to  \textit{Poisson's equation} with \textit{forcing function} $g$ if it satisfies
\begin{equation}
\Expect[\hag( \Phi_{n+1}) - \hag(  \Phi_{n})\mid \Phi_n = x]  =  -  \tilg(x)   \,, \qquad x\in\state 
\label{e:bigfish}
\end{equation}

The following results may be found in \cite{MT}:   Part (i) of the following follows from  Theorem 16.0.1,  and (ii)  follows from  Theorem 17.4.2. Part (iii) follows from Theorem 17.5.3 and the representation in \eqref{e:AsymptCLT_fish} is (17.47) of Section 17.4.3.

\begin{theorem}
\label[theorem]{t:Vuni+CLT}
Suppose that for each $\theta \in \Re^d$, $\bfPhi \eqdef \{\Phi_n \colon n \geq 0\}$ is  an  aperiodic, $\psi$-irreducible Markov chain satisfying Assumption~(A4).   Then,
\begin{romannum}
	\item
	$\bfPhi$ is $v$-uniformly ergodic with $v=e^V$:   
	there is $\varrho_v\in (0,1)$ and $b_v<\infty$ such that for any   $g\in L_v$,
	\[
	\big|
	\Expect_x  [ \tilg(\Phi_n) ]   \big|  \le    b_v \| g\|_v v(x)   \varrho_v^n    
	\]
	where the subscript $x$ on the left hand side indicates that $\Phi_0 = x$.
	
	\item  If $\| g \|_w <\infty$ with $w= v^\delta$ and $\delta\le 1$, then there is a solution to Poisson's equation
	\eqref{e:bigfish}.  	Moreover, 	$\| \hag\|_w <\infty$ and the solution can be chosen so that $\uppi(\hag) =0$.
	
	\item
	If $g\colon\state\to\Re^m$ satisfies  $g_i^2  \in L_v$ for each $i$,   then the $m$-dimensional stochastic process $\{ Z_n^g = n^{-1/2}     \sum_{k=1}^n \tilg(\Phi_k)   :  n\ge 1\}$ converges in distribution to a Gaussian $N(0,   \SigmaCLT^g)$ random variable,  and the second moment also converges:    
	\begin{equation}
		\lim_{n \to \infty} \Expect[ Z_n^g (Z_n^g)^\transpose] 
		=
		\lim_{n \to \infty} \frac{1}{n} \Cov\Big(\sum_{k=1}^n g(\Phi_k)\Big) = \SigmaCLT^g
		\label{e:AsymptCLT_phi}
	\end{equation}
	The $m\times m$ asymptotic covariance may be expressed in two equivalent forms: 
	
	\whamb  the sum of the auto-covariance matrices:
	\[
	\SigmaCLT^g \eqdef \sum_{k=-\infty}^\infty \Expect_\uppi[\tilg(\Phi_0) \tilg(\Phi_k)^\transpose]  
	\]
	in which $\{ \Phi_k  :  -\infty <k <\infty \}$ is a stationary version of the Markov chain. 
	
	\whamb  In terms of solutions to Poisson's equation  \eqref{e:bigfish}:  
	\begin{equation}
		\SigmaCLT^g = \Expect_\uppi[ \hag(\Phi_0)  \tilg(\Phi_0)^\transpose ] + \Expect_\uppi[   \tilg(\Phi_0) \hag(\Phi_0)^\transpose ] - \Expect_\uppi[ \tilg(\Phi_0)  \tilg(\Phi_0)^\transpose ]
		\label{e:AsymptCLT_fish}
	\end{equation}
\end{romannum}
\qed
\end{theorem}

\section{Noise Decomposition for General SA}
In a similar fashion  to the notation employed in \Cref{s:Markov_sec}, we use $\barh: \Re^d \to \Re^d$ to denote for each $\theta \in \Re^d$, $\barh(\theta) \eqdef \int h(\theta,x) \uppi_\theta(x)$ for the Markov chain $\bfPhi^\theta$.

Poisson's equation \eqref{e:bigfish} is the key workhorse in establishing any of the main results in the present paper. 
However,  we require a slight extension of the definition \eqref{e:bigfish}:   we are interested in solutions  with forcing functions of the form $h: \Re^d \times \state \to \Re^d$. Solutions for the joint process $\bfPsi$ will not  be required, but instead we obtain solutions for each $\theta$. The following abuse of notation will be employed throughout the paper: $\hah$ is said to be the solution to Poisson's equation with forcing function $h$ if it satisfies, for each fixed $\theta \in  \Re^d$,
\begin{equation}
\Expect[\hah(\theta, \Phi^\theta_{n+1}) - \hah(\theta, \Phi^\theta_{n})\mid \Phi^\theta_n = x] 
= 
- h(\theta, x) + \barh(\theta) 
\,, \quad x \in \state 
\label{e:fish-h}
\end{equation} 
This notation is also extended to matrix-valued functions.      

The first important example is $h = f$ that is part of the definition of the SA recursion.  The following result is \cite[Prop. 7~(ii)]{chedevborkonmey21} and is a companion to 
\Cref{t:Vuni+CLT}~(iii).

\begin{proposition}
\label[proposition]{t:bounds-H}	Suppose that  (A2) and (A4) hold. Then,    $\haf \colon\Re^d\times\state\to\Re^d$ exists solving 	\eqref{e:fish-h} with forcing function $f$, in which $\Expect_{\uppi_\theta}[\haf(\theta,\Phi^\theta_n)] =\Zero$ for each   $\theta \in\Re^d$, 
and  for a constant $b_f$  and all $\theta,\theta' \in \Re^d$ and $x \in \state$:
\begin{subequations}
	\begin{align}
		\| \haf(\theta, x)\| 
		&\leq 
		b_f\bigl(1+V(x) \bigr) \bigl[ 1+\|\theta\| \bigr] 
		\\
		\| \haf(\theta, x)  - \haf(\theta', x)  \| 
		&\leq
		b_f\bigl(1+V(x) \bigr)  \|\theta - \theta'\| 
	\end{align} 
	\label{e:Lip_haf}
\end{subequations}
\end{proposition}

\begin{corollary}
\label[corollary]{t:bounds-ha}	Suppose the assumptions of \Cref{t:bounds-H} along with (A5). Then,    $\haA \colon\Re^{d \times d}\times\state\to\Re^{d \times d}$ exists solving \eqref{e:fish-h} with forcing function $A$, in which   
$\Expect_{\uppi_\theta}[\haA(\theta,\Phi^\theta_n)] =\Zero$ for each   $\theta \in\Re^d$, 
and  for a constant $b_A$  and all $\theta,\theta' \in \Re^d$ and $x \in \state$:
\begin{subequations}
	\begin{align}
		\| \haA(\theta, x)\|_F 
		&\leq 
		b_A\bigl(1+V(x) \bigr) \bigl[ 1+\|\theta\| \bigr] 
		\\
		\| \haA(\theta, x)  - \haA(\theta', x)  \|_F 
		&\leq
		b_A\bigl(1+V(x) \bigr)  \|\theta - \theta'\| 
	\end{align} 
	\label{e:Lip_haA}
\end{subequations}
\qed
\end{corollary}

Solutions to Poisson's equation enable ``whitening'' of $\bfDelta$ through the noise decomposition of {M\'etivier} and {Priouret} \cite{metpri87},
\begin{lemma}
\label[lemma]{t:Met_Pri_Exp}
Under (A2), the representation \eqref{e:DeltaDecomp} holds with
\begin{subequations}
	\begin{align}
		\MD_{k+1} 
		& \eqdef
		\haf(\theta_k,\Phi_{k+1}) - \Expect[\haf(\theta_k,\Phi_{k+1}) \mid \clF_{k}] 
		\\
		\clT_{k+1}  
		&\eqdef
		\uppsi(\theta_{k+1},\Phi_{k+1}) 
		\\
		\Upupsilon_{k+1} 
		&\eqdef
		\frac{1}{\alpha_{k+1}} [\uppsi(\theta_{k+1},\Phi_{k+1}) - \uppsi(\theta_k,\Phi_{k+1})]
		\label{e:UpDef}
	\end{align}
	with $\uppsi(\theta,x) \eqdef f(\theta,x) - \haf(\theta,x)$ and where $\clF_k \eqdef \sigma(\theta_0, \Phi_1,\cdots,\Phi_k)$.
	\label{e:decomp_def}	
\end{subequations}
\qed
\end{lemma}
The sequence $\{\clW_{k+1}\}$ is a martingale difference sequence and $\{\Upupsilon_{k+1}\}$ is the major source of bias for stochastic approximation with constant step-size \cite{laumey23a}.
When the noise is AD, $\haf$ is only a function of $\bfPhi$, giving $\Upupsilon  \equiv 0$.

In view of the decomposition in \Cref{t:Met_Pri_Exp}, the following corollary to \Cref{t:bounds-H} is obtained,
\begin{corollary}
\label[corollary]{t:Met_Pri_bounds}
Under (A2) and (A4), we obtain the following bounds for the terms in the representation \eqref{e:DeltaDecomp}: for a constant $\bdd{t:Met_Pri_bounds}$,
\begin{subequations}
	\begin{align}
		\|\MD_{k+1}\| 
		& \leq
		\bdd{t:Met_Pri_bounds}
		[2+\Expect[  V(\Phi_{k+1}) | \clF_{k}   ]  +   V(\Phi_{k+1})] [1 + \| \theta_k \|]
		\\
		\|\clT_{k+1}\|  
		&\leq
		\bdd{t:Met_Pri_bounds}
		[1 +   V(\Phi_{k+1})] [1 + \| \theta_k \|]
		\\
		\|\Upupsilon_{k+1}\| 
		&\leq
		\bdd{t:Met_Pri_bounds} \| L \|_{v^{\frac{1}{p}}}
		[1 +   V(\Phi_{k+1})] L(\Phi_{k+1}) [1 + \| \theta_k \|]
	\end{align}	
	\label{e:Met_Pri_bounds}
\end{subequations}
\qed
\end{corollary}

\section{Moment Bounds}

\label{s:momentproofs}

The standard ODE approximation of an SA recursion begins with the introduction of   ``sampling times'' for the mean flow \eqref{e:meanflow}: $\tau_{k+1} = \tau_k + \alpha_{k+1}$, with $\tau_0 = 0$  \cite{benmetpri12,bor20a}.   Then a fixed time horizon $T>0$ is specified,  along with the sequence  $T_{n+1} = \min \{\tau_k \colon \tau_k \geq T_n + T \} $ for $n\ge 0$, initialized with $T_0 = 0$. It follows that the interval $[T_n,T_{n+1}]$ satisfies $T \leq T_{n+1} - T_n \leq T + \baralpha$ for each $n$, where $\baralpha \eqdef  \sup_k \alpha_k$.  
We let $m_n$ denote the integer satisfying $T_n = \tau_{m_n}$, and  $\{ \odestate^{(n)}_t   :  t\ge T_n \}$ denote the solution to the mean flow with initial condition $ \odestate^{(n)}_{T_n} = \theta_{m_n}$.     Convergence theory of SA proceeds by comparing  $\{\theta_k  :  m_n \le k\le m_{n+1} \}$   and 
$\{\odestate^{(n)}_{\tau_k}  :  m_n \le k\le m_{n+1} \}$.

We establish moment bounds of the form \eqref{e:Pmoment_SA} by extending the approach of \cite{chedevborkonmey21}, based on a Lyapunov function of the form
$\clV(\theta,x) = \clV^\circ(\theta) + \beta \clV^\circ(\theta)v_+(x)$ in which $\beta>0$,   and    %
\[
\clV^\circ(\theta) = 1+\| \theta \|^p  
\, , 
\quad 
v_+(x) = \Expect[ \exp(V(\Phi_{m_n+1})      )  \mid \Phi_{m_n} = x  ]   %
\, .
\]

A Lyapunov bound is obtained in the following:   
\begin{proposition}
\label[proposition]{t:LyapContract}
Suppose the assumptions of \Cref{t:BigBounds} hold.   Then for each $\rho\in(0,1)$ and $p\ge 1$ there are  constants $\varrho_0<1$,   $b_\nu$,  and $n_0\geq 1$ such that
\begin{equation}
	\Expect[\clV( \theta_{m_{n+1}}, \Phi_{m_{n+1}} )  \mid \Phi_{m_n}  ] \leq \varrho_0 \clV( \theta_{m_n}, \Phi_{m_n} ) + b_\nu
	\, , 
	\quad n\geq n_0
	\label{e:Lyap_contraction}
\end{equation}
\qed
\end{proposition}

A similar result appears in   \cite{chedevborkonmey21}   with $\rho \in (1/2,1)$ and $p=4$.    In the following we explain how these assumptions may be relaxed.

We first note that (A4) is far stronger than what is assumed in  \cite{chedevborkonmey21},  in which \eqref{e:L_is_oW} is relaxed to 
$\delta_L \eqdef  |L|_W \eqdef \sup_x   |L(x)|/W(x) < \infty$, and $\delta_L$ is assumed sufficiently small.     

The assumption that $\delta_L$ is small may be guaranteed 
under (A4) by applying 	\eqref{e:L_is_oW}:   

\begin{lemma}
\label[lemma]{t:Lemma2clt} 
If (A4) holds, then for any $\delta_0>0$ there is $r<\infty$ such that (A4) holds with $(V, W_r)$  using $W_r (x):= \max\{r,W(x)\}$,  and 
\[
\delta_L(r) \eqdef 
\sup_x \frac{1}{W_r(x)} |L(x)|  \le \delta_0  
\]
\end{lemma}

\begin{subequations}

Throughout the Appendix,  for a given $\rho\in(0,1)$  we fix  $r\ge 1$,    $T>0$, and an integer $p$  satisfying,
\begin{align}
	\| \odestate_t - \theta^* \|  & \le \half   \| \odestate_0 - \theta^* \|  \,,\qquad  t\ge T\,,  \  \odestate_0\in\Re^d
	\label{e:TAssumption}
	\\
	\max(4,  2/\rho  ) \le p   &   <   \pmax\eqdef  \frac{1}{\delta_L(r) [T+1]  }  
	\label{e:pAssumptions}
\end{align}
where in \eqref{e:TAssumption},  $\{\odestate_t : t\ge 0\}$ is a solution to the mean flow, and the bound holds for any initial condition $\odestate_0$.

Since $p$, $r$ are fixed we henceforth write $\delta_L$  (suppressing dependency on $r$).

\end{subequations} 

The following  is a simple corollary to  \cite[Prop. 12~(ii)]{chedevborkonmey21}, which considered only $p=4$.

\begin{lemma}
\label[lemma]{t:Prop12ii_CLT}
Suppose (A1)--(A4) hold.
Then, the following bound holds for $m_n < k \le  m_{n+1}$:
\begin{equation}
	(1+\|\theta_k\|)^p \leq 2^p \exp\Bigl(p\delta_L    \sum_{j=m_n+1}^{ m_{n+1} } \alpha_{j}W(\Phi_j)      \Bigr) (1+\|\theta_{m_n}\| )^p  
	\label{e:thetaMultBdd}
\end{equation}
\qed
\end{lemma}

Bounds on the conditional expectation of the right hand side of \eqref{e:thetaMultBdd} are obtained by applying \Cref{t:prop3CLT} that follows.   The proof follows from \Cref{t:Lemma2clt}, combined with an obvious extension of 
\cite[Prop. 3]{chedevborkonmey21}.

\begin{lemma}
\label[lemma]{t:prop3CLT}
Under (A4)  the following holds:      for any $n$,   any non-negative sequence $\{\delta_k : 1\le k\le n-1 \}$ satisfying $\sum \delta_k\le 1$   there is $ b_v  <\infty$  and $\delta_v>0$ such that 
\begin{equation} 
	\Expect_x\Bigl[\exp\Bigl( V(\Phi_n) +  p\delta_L  T  \sum_{k=0}^{n-1}   \delta_k W(\Phi_k)  \Bigr)  \Bigr]
	\le    b_v \exp\bigl( V(x)  - \delta_v W(x)   \bigr)    \,,    \qquad   x\in\state
	\label{e:DV3multBdd+Cor}
\end{equation}
\end{lemma}

The two lemmas combined with the Markov property imply the following under our standing assumption \eqref{e:pAssumptions}:

\begin{lemma}
\label[lemma]{t:Prop12i_CLT}
Under the assumptions of \Cref{t:Prop12ii_CLT}, the following holds for some $\delta_v>0$,   all $n\ge 0$, and  $m_n < k \le m_{n+1}$: 
\[
\Expect[\exp\bigl( V(\Phi_{k+1}) \bigr)(1 + \|\theta_k\|)^{p} \mid \Phi_{m_{n}+1}]  \leq 2^p  b_v  (1+\|\theta_{m_n}\| )^p  
\exp\bigl( V(\Phi_{m_{n} +1 })  - \delta_v  W(\Phi_{m_{n}})   \bigr) 
\]
\qed
\end{lemma}

\cite[Lemma 11]{chedevborkonmey21}  establishes the following for $p=4$.  The simple proof is easily extended to general $p\ge 4$ under (A4).

\begin{lemma}
\label[lemma]{t:contract-4}
Under  (A1)--(A4)  there exists constants $0 < \bdde{t:contract-4} < 1$, $\bdd{t:contract-4} < \infty$,    and a deterministic and vanishing sequence $\{ \bdds{t:contract-4}_n  : n\ge n_g\}$	 such that  for all $n\ge n_g$,
\begin{equation}
	\label{e:contract-4}
	\Expect \bigl[ \bigl(\| \theta_{m_{n+1}}\| + 1 \bigr)^p \mid \Phi_{m_{n}+1}   \bigr]\leq \Bigl( \bdde{t:contract-4} +  \bdds{t:contract-4}_nv(\Phi_{m_n+1}) \Bigr) \bigl(1 + \|\theta_{m_n}\| \bigr)^p  + \bdd{t:contract-4}.
\end{equation}
\end{lemma}

\begin{subequations}

\begin{proof}[Proof of \Cref{t:LyapContract}]
	
	\Cref{t:contract-4} bounds the first term in the Lyapunov function  $\clV(\theta,x) = \clV^\circ(\theta) + \beta \clV^\circ(\theta)v_+(x)$ 
	\begin{equation}
		\begin{aligned}
			\!\!\!\!\!\!
			\Expect[\clV^\circ( \theta_{m_{n+1}} )  \mid \Phi_{m_n}  ]  
			&\leq \Bigl( \bdde{t:contract-4} +  \bdds{t:contract-4}_n v_+(\Phi_{m_n}) \Bigr) \bigl(1 + \|\theta_{m_n}\| \bigr)^p  + \bdd{t:contract-4}
			\\
			& \le     \big[ \bdde{t:LyapContract}   +  \bdd{t:LyapContract}_2 \bdds{t:contract-4}_n v_+(\Phi_{m_n})  \bigr]  \clV^\circ (\theta_{m_n} ) + \bdd{t:contract-4}  +  \bdd{t:LyapContract}_1\,,
		\end{aligned}
		\label{e:LyapContractA}
	\end{equation}
	where $\bdde{t:LyapContract} \in (\bdde{t:contract-4}  , 1)$ is fixed,    and then  $\bdd{t:LyapContract} _1 =   \sup \{  \clV^\circ (\theta  )    :    
	\bdde{t:contract-4}(1 + \|\theta \| )^p   \ge  \bdde{t:LyapContract}  (1 +  \|\theta \| ^p)  \}$.     The second constant is the upper bound $  \bdd{t:LyapContract}_2 = \sup \{ (1 +x)^p(1 + x^p)^{-1} : x\ge 0 \}$.  
	
	To bound the second term in $\clV$ we first apply   the Markov property,
	\[
	\Expect[   \clV^\circ( \theta_{m_{n+1}} )v_+(\Phi_{m_{n+1}} )    \mid \Phi_{m_n}  ]   =  
	\Expect[   \clV^\circ( \theta_{m_{n+1}} )  v(\Phi_{m_{n+1} + 1} )    \mid \Phi_{m_n}  ]   
	\]
	where $v=e^V$.
	\Cref{t:Prop12i_CLT}  then gives a bound on  the conditional mean of $\clV(\theta,x) $: 
	for some finite constant $ \bdd{t:LyapContract} _3 $,
	\begin{equation} 
		\begin{aligned}
			\Expect[   \clV^\circ( \theta_{m_{n+1}} )v_+(\Phi_{m_{n+1}} ))    \mid \Phi_{m_n}  ]   
			& 	\le   2^p  b_v  (1+\|\theta_{m_n}\| )^p  
			v_+(\Phi_{m_n})  \exp\bigl( - \delta_v  W(\Phi_{m_{n}})   \bigr) 
			\\
			& \le
			2^p  b_v  \bdd{t:LyapContract}_2   \clV^\circ ( \theta_{m_n} )  v_+(\Phi_{m_n})  \exp\bigl( - \delta_v  W(\Phi_{m_{n}})   \bigr) 
			\\
			& \le  \half    \clV^\circ ( \theta_{m_n} )  v_+(\Phi_{m_n})  +  \bdd{t:LyapContract} _3 \clV^\circ ( \theta_{m_n} )  
		\end{aligned}
		\label{e:LyapContractB}
	\end{equation}
	
	The proof of the proposition is completed on combining the bounds
	\eqref{e:LyapContractA},
	\eqref{e:LyapContractB} and  choosing $\beta >0$ so that $\varrho_1\eqdef \beta   \bdd{t:LyapContract} _3 + \bdde{t:LyapContract}   <1$.
	The constant $ \varrho_0  <1 $ in \eqref{e:Lyap_contraction} is 	$ \varrho_0  = \max ( 1/2, \varrho_1)$.
\end{proof}

\end{subequations}

The following results are immediate from \eqref{e:Lyap_contraction}:
\begin{corollary}
\label[corollary]{t:Delta_bdd}
Under the assumptions of \Cref{t:BigBounds},    
\begin{romannum}
	
	\item
	$ \displaystyle
	\sup_{k \geq 0} \Expect[ \| \theta_k+1\|^p \exp(V(\Phi_{k+1}))] <\infty
	$
	
	\item there is $\bdd{t:Delta_bdd}$ depending upon $\Psi_0$ and $\rho$ such that
	\[
	\sup_{k \geq 0} \| \MD_{k} \|_p \leq \bdd{t:Delta_bdd} 
	\, , \quad 
	\sup_{k \geq 0} \| \clT_{k} \|_p  \leq \bdd{t:Delta_bdd} 
	\, , \quad 
	\sup_{k \geq 0} \| \Upupsilon_{k} \|_p  \leq \bdd{t:Delta_bdd} 
	\]
\end{romannum}
\qed
\end{corollary}

\begin{proof}[Proof of parts (i) and (ii) of \Cref{t:BigBounds}]
With \Cref{t:Delta_bdd} in hands, the proof of part (i) follows from the same arguments as the ones used in the proof of \cite[Lemma 19]{chedevborkonmey21}.

We now establish part (ii). In view of  \eqref{e:Pmoment_SA}, and applying \eqref{e:pAssumptions}, 
we have $\sum_{n=1}^\infty \Expect[ \| \tiltheta_n  \|^p]<\infty$,
implying $\tiltheta_n \overset{\as}{\to} 0$, by Fubini's theorem. 
\end{proof}

\section{Asymptotic Statistics for General Stochastic Approximation}
\label{s:GenSA_proofs}
We begin this section with a representation for the sample-path target bias, obtained by averaging both sides of \eqref{e:SA_recur} and applying the noise decomposition in \Cref{t:Met_Pri_Exp}:
\begin{lemma}
\label[lemma]{t:Targetbias_rep}
The following holds for the recursion \eqref{e:SA_recur},
\begin{subequations}
	\begin{equation}
		\frac{1}{N} \sum_{k=1}^N \barf(\theta_k) = 
		\frac{1}{N} (S^\uptau_N   - S^\Delta_{N+1} )
		\label{e:Targetbias_almost_as}
	\end{equation}
	where 
	$S^\uptau_N = \sum_{k=1}^N \frac{1}{\alpha_{k+1}}(\theta_{k+1} - \theta_k)$ and
	\begin{equation}
		S^\Delta_{N+1}  = \sum_{k=1}^N  \Delta_{k+1}
		= M_{N+1} - \clT_{N+1} + \clT_{1} - \sum_{k=1}^N \alpha_{k+1} \Upupsilon_{k+1}
		\label{e:Sum_Delta_GenSA}
	\end{equation}
	in which $\{M_{N+1} \eqdef \sum_{k=1}^N \MD_{k+1} : N\geq 1\} $ is a martingale.
\end{subequations}
\qed
\end{lemma}

The next two lemmas will be used repeatedly to establish several of the remaining results of the paper. 
\begin{lemma}[Summation by Parts]
\label[lemma]{t:Sum_parts}
For any two real-valued sequences $\{x_n,y_n: n \geq 0\}$ and integers $0 \leq N_0 < N$,
\[
\sum_{k=N_0+1}^N x_k(y_{k} - y_{k-1}) 
= x_{N+1}y_{N} - x_{N_0+1}y_{N_0} - \sum_{k=N_0+1}^{N} (x_{k+1} - x_k) y_k
\]
\qed
\end{lemma}

\begin{lemma}
\label[lemma]{t:alpha_avg}
Under (A1) the following bounds hold: for a constant $\bdd{t:alpha_avg}<\infty$
	\begin{romannum}
		\item
		$\displaystyle
		\sum_{k=1}^N \Big| \frac{1}{\alpha_{k+1}} - \frac{1}{\alpha_{k}}  \Big| \sqrt{\alpha_k} 
		\leq 
		\frac{2}{\sqrt{\alpha_0}} N^{\rho/2} + \bdd{t:alpha_avg}
		$
		\item 		
		$\displaystyle
		\sum_{k=1}^N \alpha_{k}  
		\leq 
		\frac{\alpha_0}{1-\rho} N^{1-\rho} + \bdd{t:alpha_avg}
		$
		
		\item 		
		$\displaystyle
		\sum_{k=1}^N \frac{1}{k+1}  
		\leq 
		\log(N+1) + \bdd{t:alpha_avg}
		$
		
\end{romannum}
\qed
\end{lemma}

In order to obtain moment bounds for PR averaging, \Cref{t:Targetbias_rep} tells us we need to first analyze the right hand side of \eqref{e:Targetbias_almost_as}. \Cref{t:Sum_delta_theta_bounds,t:Sum_delta_theta_bounds2} establish moment bounds for $S^\uptau$ as well as for each term in \eqref{e:Sum_Delta_GenSA}. 
\begin{lemma}
\label[lemma]{t:Sum_delta_theta_bounds}
Under the assumptions of \Cref{t:BigBounds}, there is $\bdd{t:Sum_delta_theta_bounds}$ depending upon $\Psi_0$ and $\rho$ such that the following bounds hold
\begin{romannum}
\item  $\|S^\uptau_N\|_2 
\leq \bdd{t:Sum_delta_theta_bounds} N^{\rho/2}$
\item $\|\clT_{N+1} - \clT_{1} \|_2 
\leq \bdd{t:Sum_delta_theta_bounds}$
\item $\|\sum_{k=1}^N  \alpha_{k+1} \Upupsilon_{k+1} \|_p 
\leq \bdd{t:Sum_delta_theta_bounds} N^{1-\rho}$
\end{romannum}		
\end{lemma}
\begin{proof}
To prove part (i), we apply \Cref{t:Sum_parts}, followed by the triangle inequality to obtain
\[
\begin{aligned}
\|S^\uptau_N\|_2 
&  = 
\Big\|\frac{1}{\alpha_{N+1}} \tiltheta_{N+1} - \frac{1}{\alpha_{2}} \tiltheta_{1}  
-  \sum_{k=1}^N \Big( \frac{1}{\alpha_{k+1}} - \frac{1}{\alpha_{k}} \Big)\tiltheta_{k} \Big\|_2 
\\ 
&
\leq 
\frac{1}{\alpha_{N+1}} \| \tiltheta_{N+1} \|_2 + \frac{1}{\alpha_{2}} \| \tiltheta_{1} \|_2 
+ \sum_{k=1}^N \Big| \frac{1}{\alpha_{k+1}} - \frac{1}{\alpha_{k}} \Big| \| \tiltheta_{k} \|_2
\end{aligned}
\]
Jensen's inequality along with the bounds in \Cref{t:BigBounds} gives $ \| \tiltheta_n \|_2 \leq  \| \tiltheta_n \|_p \leq \bdd{t:BigBounds}\sqrt{\alpha_n} $. An application of part (i) of \Cref{t:alpha_avg} yields the final bound.

The result in part (ii) is implied directly from \Cref{t:Delta_bdd}~(ii) and the triangle inequality.

Part (iii) is obtained as follows:  
\[
\Big\| \sum_{k=1}^N \alpha_{k+1}  \Upupsilon_{k+1} \Big\|_p
\leq 
\Big(\sup_{k \geq 0}\|\Upupsilon_{k+1} \|_p \Big) \sum_{k=1}^N|\alpha_{k+1}|
\]
and \Cref{t:Delta_bdd}~(ii) gives $\sup_{k \geq 0}\|\Upupsilon_{k+1} \|_p \leq \bdd{t:Delta_bdd}$. An application of part (ii) of \Cref{t:alpha_avg} finishes the proof.
\end{proof}

The following notation is adopted to define parameter independent disturbance processes: let $M^*_{N+1} = \sum_{k=1}^N \clW^*_{k+1} $ with
\begin{equation}
\clW^*_{k+1} 
\eqdef
\haf(\theta^*,\Phi_{k+1}) - \Expect[\haf(\theta^*,\Phi_{k+1}) \mid \clF_{k}] 
%
\label{e:parameter_indep}
\end{equation}
The above martingale difference sequence is the ``noise'' sequence defining the optimal asymptotic covariance of SA \eqref{e:Cov_lowerbound_def}.

\begin{lemma}
\label[lemma]{t:Sum_delta_theta_bounds2}
Under the assumptions of \Cref{t:BigBounds}, there is $\bdd{t:Sum_delta_theta_bounds2}$ depending upon $\Psi_0$ and $\rho$ such that the following bounds hold
\begin{romannum}
\item  $\| M_{N+1} - M^*_{N+1}  \|_2 
\leq 
\bdd{t:Sum_delta_theta_bounds2} N^{(1 - \rho)/2}$
\item $\| M^*_{N+1}  \|_2 
\leq  
\bdd{t:Sum_delta_theta_bounds2} N^{1/2}$
\end{romannum}		
\end{lemma}
\begin{proof}
For part (i) we use the martingale difference property to obtain
\[
\|M_{N+1} - M^*_{N+1}\|^2_2 
= \Big \|  \sum_{k=1}^N  \MD_{k+1} - \MD^*_{k+1}   \Big\|^2_2   
=  \sum_{k=1}^N \| \MD_{k+1} - \MD^*_{k+1} \|^2_2
\]
From \Cref{t:bounds-H} and \Cref{t:BigBounds}, we have that $\|  \MD_{k+1} - \MD^*_{k+1} \|^2 \leq \bdd{t:bounds-H}\| \tiltheta_n \|^2_2 \leq \bdd{t:bounds-H} \bdd{t:BigBounds} \alpha_n $ for some constant $\bdd{t:Sum_delta_theta_bounds2}$ depending upon $\Psi_0$. An application of  part (ii) of \Cref{t:alpha_avg} yields (i).

The proof of part (ii) follows similarly to (i): since $\{\MD^*_{k+1}\}$ is a martingale difference sequence,
\[
\| M^*_{N+1}\|^2_2  =  \sum_{k=1}^N \| \MD^*_{k+1} \|^2_2
\leq  
N \Big( \sup_{k \geq 0}\|\MD^*_{k+1} \|_2^2 \Big) 
\leq 
\bdd{t:Delta_bdd} N
\]
where the last bound is obtained from an application of \Cref{t:Delta_bdd}~(ii).
\end{proof}

We are now equipped to prove part (iii) of \Cref{t:BigBounds}. Only the proof for the  multiplicative (MU) noise case is given. The result for when the noise is AD  follows almost identically to the MU setting with the only difference being that $\Upupsilon \equiv 0$, as explained after \Cref{t:Met_Pri_Exp}.

\begin{proof}[Proof of part (iii) of \Cref{t:BigBounds}]
A Taylor series expansion of $\barf(\theta_k)$ around $\theta^*$ gives
\begin{subequations}
\begin{equation}
	\barf(\theta_k) = \barf(\theta^*) + A^* (\theta_k - \theta^*) + \clE^T_k
\end{equation}
where under the assumptions of the theorem, the second order Taylor series expansion error $\clE^T$ admits the upper bound:
\begin{equation}
	\| \clE^T_k \|_p
	\leq 
	L_A\| \tiltheta_k \|^2_p 
	\leq 
	L_A  \bdd{t:BigBounds} \alpha_k
	\label{e:TS_bound}
\end{equation}
\label{eq:TS_approx}
\end{subequations}
in which $L_A$ is the Lipschitz constant associated with $\barA$.

In view of \eqref{eq:TS_approx}, \eqref{e:Targetbias_almost_as} is equivalently expressed as
\[
\thetaPR_N - \theta^*
= 
[A^*]^{-1} \frac{1}{N} \Big[ S^\uptau_N   - S^\Delta_{N+1} + \sum_{k=1}^N \clE^T_k   + M^*_{N+1}    - M^*_{N+1}\Big] 
\]

Taking $L_2$ norms on both sides and applying the triangle inequality we obtain,
\[
\begin{aligned}
\| \tilthetaPR_N  \|_2 
&\leq
\|[A^*]^{-1}\|_F \frac{1}{N} 
\Big(  
\|S^\uptau_N \|_2 + \| S^\Delta_{N+1} - M^*_{N+1}  \|_2 
+ \sum_{k=1}^N \| \clE^T_k\|_2   
+\| M^*_{N+1}\|_2  
\Big)   
\\
&\leq 
\|[A^*]^{-1}\|_F \frac{1}{N} 
(\epsy_N^a  + \epsy_N^b + \epsy_N^c)
\end{aligned}
\]
in which $\| \varble \|_F$ denotes the Frobenius norm and the identity \eqref{e:Sum_Delta_GenSA} gives
\[
\begin{aligned}
\epsy_N^a &= \|S^\uptau_N \|_2 +\| M_{N+1} - M^*_{N+1}\|_2 + \|\clT_{N+1} + \clT_{1}\|_2 
\\
\epsy_N^b &= \Big\| \sum_{k=1}^N \alpha_{k+1} \Upupsilon_{k+1} \Big\|_2 + \sum_{k=1}^N \| \clE^T_k\|_2
\\
\epsy_N^c &= \|M^*_{N+1}\|_2
\end{aligned}
\]
\Cref{t:Sum_delta_theta_bounds,t:Sum_delta_theta_bounds2} and \eqref{e:TS_bound} yield the following bounds: for a constant $\bdd{t:BigBounds}$ depending  upon $\Psi_0$ and $\rho$,
\[
\epsy_N^a \leq \bdd{t:BigBounds} \max{\{N^{\rho/2}, N^{(1-\rho)/2}\}} 
\, , 
\quad 
\epsy_N^b \leq \bdd{t:BigBounds} N^{1-\rho}
\,,
\quad 
\epsy_N^c \leq \bdd{t:BigBounds} N^{1/2}
\]
The error is dominated by $\epsy_N^b$ when $\rho<1/2$ and by  $\epsy_N^c$ when $\rho>1/2$, completing the proof.
\end{proof}

We conclude this section with a proof of the identity in \eqref{e:target_bias_lim} for the average target bias:  let $\haA(\theta,x) = \partial_\theta \haf(\theta,x)$, $\partial_\theta \uppsi(\theta,x) \eqdef A(\theta,x) - \haA(\theta,x) $ and
\begin{equation}
\Upupsilon^*_{k+1} \eqdef \partial_\theta \uppsi(\theta^*,\Phi_{k+1}) f(\theta^*,\Phi_{k+1}) 
\, , \quad 
\barUpupsilon^* = \Expect[ \Upupsilon^*_{k}]
\label{e:Upstar}
\end{equation}

\begin{lemma}
\label[lemma]{t:Up_haAf}
Suppose (A1)--(A5) hold. Then, for each $\rho \in (0,1)$ there is $\bdd{t:Up_haAf}$ depending upon $\Psi_0$ and $\rho$,
\begin{romannum}
\item  $\displaystyle \Upupsilon_{k+1}  = \partial_\theta \uppsi(\theta_{k}, \Phi_{k+1}) f(\theta_{k}, \Phi_{k+1}) +  \bdd{t:Up_haAf} \alpha_{k+1}$

\item $\displaystyle  \| \Upupsilon_{k+1} - \Upupsilon^*_{k+1}   \|_2 \leq \bdd{t:Up_haAf} \alpha^{1/2}_{k+1}$
\end{romannum}
\end{lemma}
\begin{proof}
A Taylor series expansion of $\haf$ around $\theta_k$ gives  for each $k$
\begin{equation}
\uppsi(\theta_{k+1}, \Phi_{k+1}) - \uppsi(\theta_{k}, \Phi_{k+1})  = 
\partial_\theta \uppsi(\theta_{k}, \Phi_{k+1}) (\theta_{k+1} -\theta_{k}) + \clE^{T}_k
\label{e:Taylor}
\end{equation}
where $\clE^{T}$ denotes the approximation error. By the mean value theorem, there is $\theta^\circ \in (\theta_{k},\theta_{k+1})$ such that
$
\clE^{T}_k = [\partial_\theta \uppsi(\theta^\circ, \Phi_{k+1})-\partial_\theta \uppsi(\theta_{k}, \Phi_{k+1}) ]  (\theta_{k+1} -\theta_{k})
$.
\Cref{t:bounds-ha} implies the upper bound: $\| \clE^{T}_k \|_2  \leq  \bdd{t:Up_haAf} \| \theta_{k+1} -\theta_{k} \|_2^2$ with $\bdd{t:Up_haAf}$ depending upon $\Psi_0$ and $\rho$.

Using the identity \eqref{e:NoisyEuler} and Lipschitz continuity of $\barf$, we obtain, through the triangle inequality,
\[
\| \theta_{k+1} -\theta_{k} \|_2 = \alpha_{k+1} \| \barf(\theta_k) + \Delta_{k+1} \|_2
\leq  \alpha_{k+1} (L_{\bar{f}} \| \tiltheta_k\|_2 + \|\Delta_{k+1} \|_2) \leq \alpha_{k+1} \bdd{t:Up_haAf}
\]
in which $L_{\bar{f}}$ is the Lipschitz constant associated with $\barf$ and the last bound follows from \eqref{e:Pmoment_SA} and \Cref{t:Delta_bdd}~(ii), for a potentially larger $\bdd{t:Up_haAf}$.

Applying \eqref{e:SA_recur} to the right side of \eqref{e:Taylor} completes the proof of part (i), in view of  the definition of $\Upupsilon$ in \eqref{e:UpDef}.

The result in part (ii) follows from Lipschitz continuity of $\haA$ and $A$, \eqref{e:Pmoment_SA}, and part (i), via the triangle inequality.
\end{proof}

\begin{proof}[Proof of part (iv) of \Cref{t:BigBounds}]
The first step is to take expectations of both sides of \eqref{e:Targetbias_almost_as}. From \eqref{e:DeltaDecomp}, we obtain $\Expect[\Delta_{k+1}] = \Expect[-\clT_{k+1} + \clT_{k} - \alpha_{k+1} \Upupsilon_{k+1}]$, which in turn gives,
\[
\frac{1}{N} \sum_{k=1}^N \Expect[\barf(\theta_k)]= \frac{1}{N}\Big( \Expect[S^\uptau_N] + \Expect[\clT_{N+1} - \clT_{1} ] + \sum_{k=1}^N  \alpha_{k+1}  \Expect[ \Upupsilon_{k+1}] \Big)
\]
Adding and subtracting $\frac{1}{N}\sum_{k=1}^N  \alpha_{k+1} \barUpupsilon^*$ to the right side of the above equation with $\barUpupsilon^*$ defined in \eqref{e:Upstar} yields
\begin{equation}
\frac{1}{N} \sum_{k=1}^N \Expect[\barf(\theta_k)] =  \frac{1}{N}\Big( \barUpupsilon^* \sum_{k=1}^N [\alpha_{k+1}  +  \epsy^a_k ]  + \epsy^b_N \Big)
\label{e:almosttarget}
\end{equation}
in which
\[
\begin{aligned}
\epsy^a_k&=  \alpha_{k+1} \Expect[ \Upupsilon_{k+1} -  \Upupsilon^*_{k+1}] 
\\
\epsy^b_N  &=  \Expect[S^\uptau_N] + \Expect[\clT_{N+1} - \clT_{1} ] 
\end{aligned}
\]
Bounds on the terms constituting $\epsy^b$ are given by \Cref{t:Sum_delta_theta_bounds} and are as follows:
\begin{equation}
\begin{aligned}
	\|  \Expect[S^\uptau_N] \| 
	\leq \|S^\uptau_N \|_2  
	\leq \bdd{t:Sum_delta_theta_bounds} N^{\rho/2}
	\\
	\|  \Expect[\clT_{N+1} - \clT_{1} ] \| 
	\leq   \| \clT_{N+1} -  \clT_{1} \|_2     
	\leq  \bdd{t:Sum_delta_theta_bounds}
\end{aligned}
\end{equation}
Consequently, $ \| \epsy^b_N  \| \leq \bdd{t:Sum_delta_theta_bounds} N^{\rho/2} $.
A bound on $\epsy^a$ is achieved from \Cref{t:Up_haAf}:
\[
\|\epsy^a_k\| \leq |\alpha_{k+1}|  \| \Expect[ \Upupsilon_{k+1} -  \Upupsilon^*_{k+1} ] \|
\leq  
|\alpha_{k+1}| \| \Upupsilon_{k+1} -  \Upupsilon^*_{k+1} \|_2 \leq \bdd{t:Up_haAf} \alpha^{3/2}_{k+1}
\] 
where the constants $\bdd{t:Sum_delta_theta_bounds}$ and $\bdd{t:Up_haAf}$ depend upon $\Psi_0$ and $\rho$.

Using the upper bounds for $\epsy^a$ and $\epsy^b$, we apply \Cref{t:alpha_avg}~(ii) to the right side of \eqref{e:almosttarget} and divide both sides by $\alpha_{N+1}$ to obtain
\begin{equation}
\frac{1}{\alpha_{N+1}}\sum_{k=1}^N \frac{1}{N} \Expect[   \barf(\theta_k) ] = \frac{1}{1-\rho} \barUpupsilon^* + \epsy^c_N
\label{e:targetbias_rate}
\end{equation}
in which
\[
\|\epsy^c_N\| \leq  \bdd{t:stats_summary} \max\{ N^{-\rho/2} ,  N^{3\rho/2 - 1}\}
\]
where $\bdd{t:stats_summary}$ is a constant depending upon $\Psi_0$ and $\rho$. The norm of $\epsy^c_N$ is dominated by $\epsy^a_N$ for $\rho <1/2$ and by $\epsy^b_N$ for $\rho>1/2$. Moreover, $\epsy^b_N$ also dominates the target bias for $\rho>2/3$.

Taking limits of both sides of the above equation yields \eqref{e:target_bias_lim} for $\rho \in (0,2/3)$.

\end{proof}


\def\hau{\hat{u}}
\def\haDelta{\hat{\Delta}}

\section{Noise Decomposition for Linear SA}
\label{s:more_MetPri}

In general, the representation for $\bfDelta$ in \Cref{t:Met_Pri_Exp} does not appear to be enough to obtain the asymptotic covariance for recursions with multiplicative noise, and $\rho < 1$. Restricting to linear recursions allows us to obtain the finer representation  for $\bfDelta$ in  \eqref{e:Deltadecomp_recur_intro}. This expression is obtained based on recursive decompositions of stochastic processes  
that are affine in the parameter.   One example   is $\Upupsilon_{k+1} = (A_{k+1}-\haA_{k+1}) (A_{k+1} \theta_k + b_{k+1})$,  $k\ge 0$  (recall   \eqref{e:upsilon_lin}).    

The next lemma relies on solutions to Poisson's equation as in \Cref{t:Met_Pri_Exp} to define the main step in this recursion.
Its proof is identical to  the proof of  \eqref{e:DeltaDecomp}.

\begin{lemma}
\label[lemma]{t:genG_Met_Decomp}
Suppose the matrix valued function $M\colon\state\to\Re^{d\times d}$ and vector valued function $u\colon\state\to\Re^{ d}$ defining $G(\Psi_k) = M(\Phi_{k+1}) \theta_k   +   u(\Phi_{k+1}) $ satisfy  $\| M (x) \|_F + \| u (x) \|\leq \bdd{t:genG_Met_Decomp} v_+^{\epsy}(x)$ for $\epsy>0$ sufficiently small, and with respective means $\barM$, $\baru$. Let  $\haM\colon\state\to\Re^{d\times d}$ and  $\hau \colon\state\to\Re^{ d}$  denote zero-mean solutions to   Poisson's equation \eqref{e:bigfish}:
\[
\Expect[\haM( \Phi_{k+1}) - \haM(  \Phi_{k})\mid \Phi_k = x]  =  M(x) - \barM \qquad 
\Expect[\hau( \Phi_{k+1}) - \hau(  \Phi_{k})\mid \Phi_k = x]  =  u(x) - \baru   
\]
Then,  denoting  $\haG( \theta,x) = \haM(x)\theta+\hau(x)$ and $\uppsi^G(\theta, x)  = G(\theta,x) - \haG(\theta,x)$,  we have for $k\ge 0$, 
\begin{equation}
\begin{aligned}
	G(\Psi_k) &= M(\Phi_{k+1}) \theta_k   +   u(\Phi_{k+1}) 
	\\
	&  = \MD^G_{k+1} -  \clT^G _{k+1} +  \clT^G _{k} - \alpha_{k+1} \Upupsilon^G_{k+1} + \barG (\theta_k)   
\end{aligned}
\label{e:Gdecomp}
\end{equation}
where  $\barG (\theta_k)   =  \barM \theta_k + \baru$,    and for a constant  $\bdd{t:genG_Met_Decomp}$  independent of $\Psi_0$, 
\whamb
$\{ \MD^G_{k+1} \eqdef \haG( \theta_k,\Phi_{k+1}) - \Expect[\haG( \theta_k,\Phi_{k+1}) \mid \clF_{k}]  \}$ is a martingale difference sequence satisfying  $\Expect[\|\MD^G_{k+1} \|^p \mid \Psi_k = z]   \leq \bdd{t:genG_Met_Decomp} \clV(z) $.

\whamb $
\displaystyle
\clT^G_{k+1} \eqdef \uppsi( \theta_{k+1},\Phi_{k+1})
$ with $\Expect[\|\clT^G_{k+1} \|^p \mid \Psi_k = z]   \leq \bdd{t:genG_Met_Decomp} \clV(z) $.

\whamb $
\displaystyle
\Upupsilon^G_{k+1} \eqdef \frac{1}{\alpha_{k+1}} [\uppsi^G(\theta_{k+1}, \Phi_{k+1}) - \uppsi^G(\theta_{k}, \Phi_{k+1})]
$ with $\Expect[\|\Upupsilon^G_{k+1}\|^p \mid \Psi_k = z]   \leq \bdd{t:genG_Met_Decomp} \clV(z) $.
\qed
\end{lemma}

%
%
%
%

\begin{subequations}

\wham{Construction of the disturbance decomposition}

The  representation in \eqref{e:Deltadecomp_recur_intro} is obtained by applying \Cref{t:genG_Met_Decomp}  recursively, to construct
a sequence of stochastic processes of the form,
\begin{align}
G^{(i)}(\Psi_k) &= M^{(i)} (\Phi_{k+1})\theta_k + u^{(i)}(\Phi_{k+1}) 
\label{e:Gdecomp-def}
\\
&  = \clW^{(i)} _{k+1} -  \clT^{(i)} _{k+1} +  \clT^{(i)}_{k} - \alpha_{k+1}  G^{(i+1)}(\Psi_k) + \barG^{(i)} (\theta_k)   
\label{e:Gdecomp-i}
\end{align}
For each $i$, we denote   
\begin{align}
\MD^{(i)}_{k+1} = \MD^G_{k+1}
\, , \quad 
\clT^{(i)}_{k+1} = \clT^G_{k+1}
\, , \quad 
G^{(i+1)}(\Psi_k) 
= \Upupsilon^{G}_{k+1}
\label{e:Gdecomp-one_more}
\end{align}	
in which these terms are defined as in \Cref{t:genG_Met_Decomp} with $G(\Psi_k) = G^{(i)}(\Psi_k)$.
\label{e:Grecur-i}
\end{subequations}

The recursion is initialized using  
\[
\begin{aligned}
G^{(0)}(\Psi_k) & = f(\Psi_k)= M^{(0)} (\Phi_{k+1})\theta_k + u^{(0)}(\Phi_{k+1})  
\\
\barG^{(0)}(\theta_k) &= \barf(\theta_k)= \barM^{(0)}\theta_k + \baru^{(0)} 
\\
\text{where }  &M^{(0)} (\Phi_{k+1}) = A_{k+1} \, , \quad u^{(0)}(\Phi_{k+1})  = -  b_{k+1}  
\\
&\barM^{(0)} = A^* \, , \qquad \qquad \quad  \baru^{(0)} = -  \barb  
\end{aligned} 
\]

The subsequent functions composing the sequence of affine functions $\{G^{(i)}\}$ in \eqref{e:Grecur-i} are defined through the following steps:

\whamb
For $1\le i\le m$:      Given   $M^{(i)}(\Phi_{k+1}) $ and $ u^{(i)}(\Phi_{k+1})$ that define $G^{(i)}(\Psi_k)$  via \eqref{e:Gdecomp-def},   apply 
\Cref{t:genG_Met_Decomp} with $G(\Psi_k) = G^{(i)}(\Psi_k)$ to obtain the terms $\{\MD^{(i)}_{k+1} , \clT^{(i)}_{k+1}, G^{(i+1)}(\Psi_k)   \}$ in \eqref{e:Gdecomp-i}.

%

%

%

\begin{proposition}
\label[proposition]{t:Met_Recursively}
Subject to the assumptions of \Cref{t:stats_summary},  and for any fixed $m\ge 1$, 
the representation \eqref{e:Deltadecomp_recur_intro} holds
for each $k\ge 0$. Moreover, the terms in \eqref{e:Deltadecomp_recur_intro} satisfy the following for a constant $\bdd{t:Met_Recursively}$ depending upon $\Psi_0$ and $\rho$:  
\begin{romannum}
\item 
$\sup_k \|  \Upupsilon^{\sbullet}_{k+1}  \|_2 \le \bdd{t:Met_Recursively}$

\item  $\{
\MD^\sbullet_{k+2}  :  k\ge 0\}$ is a  martingale difference sequence satisfying 
$\sup _k \|  \MD^{\sbullet}_{k+1}  \|_2 \le \bdd{t:Met_Recursively}$

\item  
$\sup _k \|  \clT^{\sbullet}_{k+1}  \|_2 \le \bdd{t:Met_Recursively}$,  and
$\|  \clT^\sbullet_{k+1} -  \clR^\sbullet_{k+1} \|_2   \le \bdd{t:Met_Recursively} \alpha_{k+1}/k$  for   $k\ge 0$

\item  The deterministic sequences of matrices  $ \{\barG_{k+1}\}  $ and vectors $\{ \baru_{k+1} \}$ are convergent, and hence uniformly bounded in $k$.   
Moreover,
\[
\|  \alpha_{k+1} \barG_{k+1}  - \alpha_k \barG_{k} \|_F  \le \bdd{t:Met_Recursively}    \alpha_k/k \,,  \qquad k\ge 1
\]
\end{romannum}
The special case $m=0$ yields \eqref{e:DeltaDecomp} and  $
\MD^{(0)}_{k+1}  = \MD_{k+1} $, $
\clT^{(0)}_{k+1} = \clT_{k+1}  $ in which $\{\MD_k\}$ and $\{\clT_k\}$ are defined in \eqref{e:decomp_def}.
\end{proposition}
\begin{proof}
Following the steps outlined after \eqref{e:Gdecomp-i}, we obtain $\eqref{e:Deltadecomp_recur_intro}$ in which, 
\begin{equation}
\begin{aligned} 
	\MD^\sbullet_{k+1}   &=  
	\sum^{m}_{i=0}\beta_{k+1}^{i}   \clW^{(i)}_{k+1} 
	\qquad \qquad   \qquad \qquad  
	\Upupsilon^{\sbullet}_{k+1}  =  G^{(m+1)}(\Psi_k)
	\\
	\alpha_{k+1} \barG_{k+1}\tiltheta_k  &=   \sum^{m}_{i=1} \beta_{k+1}^{i}  \barM^{(i)} \tiltheta_k
	\qquad   \qquad \qquad \qquad 
	\clT^\sbullet_{k+1}  = 
	\sum^{m}_{i=0}\beta_{k+1}^{i} \clT^{(i)}_{k+1} 
	\\
	\alpha_{k+1}\upbeta^\circ_{k+1} & =  -  \sum^{m}_{i=1} \beta_{k+1}^{i} [\baru^{i} + \barM^{(i)} \theta^*]
	\qquad \qquad \quad 
	\clR^\sbullet_{k}  = \clT^\sbullet_{k} +
	\sum^{m}_{i=0}(\beta_{k}^{i}  - \beta_{k-1}^{i})   \clT^{(i)}_{k}  
\end{aligned}
\label{e:Deltadecomp_recur_revisited}
\end{equation}
where $\beta^i_k = (-\alpha_{k})^i $ and $\barM^{(i)},\baru^{(i)}$ define $\barG^{(i)}(\theta) = \barM^{(i)} \theta + \baru^{(i)}$ for each $\theta \in \Re^d$ and $i$.

The moment bounds in parts (i)--(iii) are immediate from the bounds given as results of \Cref{t:genG_Met_Decomp}. The coupling between $\clT^\sbullet_k$ and $\clR^\sbullet_k$ in part (iii) is a consequence of the definitions in \eqref{e:Deltadecomp_recur_revisited} and the triangle inequality: for a constant $
\bdd{t:Met_Recursively}$ depending upon $\Psi_0$ and $\rho$,
\[
\| \clR^\sbullet_{k}  - \clT^\sbullet_{k} \|_2 \leq  \sum^{m}_{i=1} |\beta_{k}^{i}  - \beta_{k-1}^{i}|   \| \clT^{(i)}_{k}\|_2 
\leq  \bdd{t:Met_Recursively} N^{-1-\rho}
\]
The last bound holds since the term associated with $i=1$ dominates.

It remains to prove (iv). For each $\theta$, the following is immediate from the definitions in \eqref{e:Deltadecomp_recur_revisited}:
$
\lim_{k \to \infty} \barG_{k+1}[\theta - \theta^*] - \upbeta^\circ_{k+1}
=  - \barG^{(1)}(\theta)  
$.	
Moreover, by the triangle inequality,
\[
\| 	\alpha_{k+1} \barG_{k+1} - 	\alpha_{k} \barG_{k}\|_F \leq \sum^{m}_{i=1} |\beta_{k+1}^{i}  - \beta_{k}^{i}|   \| \barM^{(i)}\|_F \leq  \bdd{t:Met_Recursively} N^{-1-\rho}
\]
\end{proof}

Analogously to the definition of $\{\MD^*_{k}\}$ in \eqref{e:parameter_indep}, we adopt the notation $\{ \MD^{\sbullet*}_{k} , \clW^{(i)*}_{k}   \}$ to define parameter independent disturbance processes. Moreover, we denote 
\begin{equation}
H_{N+1} = \sum_{k=1}^N \MD^{\sbullet}_{k+1} \, , \quad H^*_{N+1} = \sum_{k=1}^N \MD^{\sbullet*}_{k+1}
\label{e:HN}
\end{equation}

We conclude this subsection by obtaining moment bounds on the terms defining $\{\MD^{\sbullet}_{k+1}\}$ in \eqref{e:Deltadecomp_recur_revisited}. %
\begin{lemma}
\label[lemma]{t:Markov_Deltarecur_bounds2}
Suppose that the assumptions of \Cref{t:Met_Recursively} hold. Then, if $m$ in \eqref{e:Deltadecomp_recur_intro} is chosen such that $m\rho>1$ for a fixed $\rho \in (0,1)$, there exists $\bdd{t:Markov_Deltarecur_bounds2}$ depending upon $\Psi_0$ and $\rho$ such that the following bounds hold, for $0 \leq i \leq   m$:
\begin{romannum}					
\item $\|\sum_{k=1}^N   \beta^i_{k+1}( \clW^{(i)}_{k+1} -  \clW^{(i)*}_{k+1} )\|_2  
\leq  \bdd{t:Markov_Deltarecur_bounds2} \max\{ N^{(1-[2i+1]\rho)/2} , \sqrt{\log(N)}\}$

\item $\|\sum_{k=1}^N   \beta^i_{k+1} \clW^{(i)*}_{k+1} \|_2 
\leq   \bdd{t:Markov_Deltarecur_bounds2} \max\{ N^{(1-i\rho)/2}, \sqrt{\log(N)} \}$


\end{romannum}	
\end{lemma}
\begin{proof}
Parts (i) and (ii) follow similarly to \Cref{t:Sum_delta_theta_bounds2}: using the martingale difference property, we have from \Cref{t:Delta_bdd}~(i) that there is a constant $\bdd{t:Markov_Deltarecur_bounds2}$ depending upon $\Psi_0$ and $\rho$ such that
\[
\begin{aligned}
\Big \|  \sum_{k=1}^N  \beta^i_{k+1}( \clW^{(i)}_{k+1} -  \clW^{(i)*}_{k+1} ) \Big\|^2_2  
&= 
\sum_{k=1}^N |\beta_{k+1}^{2i} | \|  \clW^{(i)}_{k+1} -  \clW^{(i)*}_{k+1}  \|^2_2 
\leq 
\sum_{k=1}^N \bdd{t:Markov_Deltarecur_bounds2} |\beta_{k+1}^{2i} |
\| \tiltheta_n \|^2_2  
\\
&\leq \sum_{k=1}^N \bdd{t:Markov_Deltarecur_bounds2}  \max\{ (k+1)^{-(2i+1)\rho} , (k+1)^{-1} \}
\\
\Big\|\sum_{k=1}^N  \beta_{k+1}^{i} 
\clW^{(i)*}_{k+1}  \Big\|^2_2   
&=
\sum_{k=1}^N  |\beta_{k+1}^{i}|^2 
\|\clW^{(i)*}_{k+1}  \|^2_2
\leq \bdd{t:Markov_Deltarecur_bounds2} \sum_{k=1}^N  |\beta_{k+1}^{2i}| 
\\
&\leq  \bdd{t:Markov_Deltarecur_bounds2} \sum_{k=1}^N \max\{ (k+1)^{-2i\rho } ,(k+1)^{-1}  \}
\end{aligned}
\]
\Cref{t:BigBounds} and \Cref{t:alpha_avg}~(ii) establish the final bounds for a potentially larger $\bdd{t:Markov_Deltarecur_bounds2}$.
\end{proof}

The martingales $\{ H_{N+1}  ,  H^*_{N+1} \}$  defined in  \eqref{e:HN} admit attractive bounds:

\begin{corollary}
\label[corollary]{t:Hsolidarity}
Under the assumptions of \Cref{t:Markov_Deltarecur_bounds2}, the following bounds hold for a constant $\bdd{t:Hsolidarity}$ depending upon $\Psi_0$ and $\rho$: 
\[
\| H_{N+1} - H_{N+1}^* \|_2 \leq \bdd{t:Hsolidarity} \max\{ N^{(1-\rho)/2} , \sqrt{\log(N)} \} 
\]
\qed
\end{corollary}

\section{Asymptotic Statistics for Linear Stochastic Approximation}
\label{s:LinSA_proofs}
In the special case of linear SA, obtaining a representation for the target bias directly translates to a bias representation for PR averaged estimates because linearization of $\barf$ is no longer needed: $\barf(\theta) = A^*\tiltheta$. The next lemma is a version of \Cref{t:Targetbias_rep} for PR averaging.
\begin{lemma}
\label[lemma]{t:PRest_rep}
The following holds for the PR averaged estimate \eqref{e:linearSA},
\begin{equation}
A^*\tilthetaPR_N = 
\frac{1}{N} (S^\uptau_N   -  S^\Delta_{N+1} )
\label{e:PRest_almost_as}
\end{equation}
where 
$S^\uptau_N$ and $S^{\Delta}_N$ are defined exactly as in \Cref{t:Targetbias_rep}.
\qed
\end{lemma}

\subsection{Bias} 
\label{s:Bias_proofs}
We begin by analyzing bias within the AD noise setting.

\begin{lemma}
\label[lemma]{t:Lyapunov_eq_norm}
Suppose that the sequence $\{\thbias_n\} \subseteq \Re^d$ satisfies the recursion:
\begin{equation}
\thbias_{n+1} = (I+\alpha_{n+1} A^*) \thbias_n + \alpha_{n+1}\epsy_{n+1}
\label{e:GenRecurr_beta}
\end{equation}
where $A^*$ is Hurwitz, $\{\alpha_n\}$ is as defined by (A1) and $\{\epsy_{n+1}\}$ is a deterministic sequence.
Then, for a constant $\bdd{t:Lyapunov_eq_norm}$ and $n_b>0$ sufficiently large,
\[
\| \thbias_{n+1} \|_{\Lyapsol} \leq (1- \bdd{t:Lyapunov_eq_norm}\tfrac{1}{2}\alpha_{n+1})\| \thbias_{n} \|_{\Lyapsol}
+ \|\epsy_{n+1}\|_{\Lyapsol} \, , \quad n \geq n_b
\]
in which  $\Lyapsol>0$  is the unique positive definite matrix solving the Lyapunov equation, $[A^*]^\transpose \text{\Lyapsol} + \text{\Lyapsol} [A^*] = - I$.
\end{lemma}
\begin{proof}
The triangle inequality gives
\[
\| \thbias_{n+1} \|_{\Lyapsol} \leq \| (I+\alpha_{n+1} A^*) \thbias_n  \|_{\Lyapsol} +   \alpha_{n+1} \|\epsy_{n+1}\|_{\Lyapsol}	
\]
Letting $\lambda_1$ and $\lambda_d$ denote the eigenvalues of $\Lyapsol$ with maximal and minimal real part, respectively, we obtain the upper bound: for a constant $\bdd{t:Lyapunov_eq_norm}<\infty$,
\[
\begin{aligned}
\| (I+\alpha_{n+1} A^*) \thbias_{n} \|^2_{\Lyapsol}
&=\thbias^\transpose_{n} (I+\alpha_{n+1} A^*)^\transpose {\Lyapsol} (I+\alpha_{n+1} A^*) \thbias_n 
\\
& =	\| \thbias_{n} \|^2_{\Lyapsol}  - \alpha_{n+1} \| \beta_{n} \|^2_{\Lyapsol}
+ \alpha^2_{n+1} \|A^* \beta_{n} \|^2_{\Lyapsol}
\\
&\leq \| \thbias_{n} \|^2_{\Lyapsol} \Big(1 - \frac{1}{\lambda_1}\alpha_{n+1} + \alpha^2_{n+1} \bdd{t:Lyapunov_eq_norm} \frac{\lambda_1}{\lambda_d} \Big)
\end{aligned}
\]
Since ${\Lyapsol}$ is positive definite, an application of the bound $\sqrt{1-\delta} \leq 1 - \tfrac{1}{2}\delta$ with $\delta < \frac{1}{\lambda_1}$ completes the proof for $n\geq n_b$ with $n_b$ sufficiently large.
\end{proof}

\begin{proof}[Proof of part (i) of \Cref{t:stats_summary}]
Applying the identity $\barf(\theta) = A^*\tiltheta$ to \eqref{e:NoisyEuler}, taking expectations of both sides and rearranging terms gives
\[
\Expect[\tiltheta_{n+1}] = (I+\alpha_{n+1} A^*) \Expect[\tiltheta_n] + \alpha_{n+1} \Expect[\Delta_{n+1}]
\]
An application of \Cref{t:Lyapunov_eq_norm} with $\thbias_n =\Expect[\tiltheta_n] $ and $\epsy_n = \Expect[\Delta_{n}]$, yields the contraction: for $n_b>0$ sufficiently large
\[
\begin{aligned}
\clE_{n+1} 
&\leq (1 -  \bdd{t:Lyapunov_eq_norm}\tfrac{1}{2}\alpha_{n+1}) \clE_{n} + \alpha_{n+1} \clE^\Delta_n \, , \quad n\geq n_b
\\
\text{with }
\clE_n &\eqdef \| \thbias_{n} \|_{\Lyapsol}
\, , \quad 
\clE^\Delta_n \eqdef \| \epsy_n \|_{\Lyapsol}
\end{aligned}
\]
By induction, we obtain the following upper bound: for each $n\geq n_b$,
\[
\begin{aligned}
\clE_n \leq  \Xi_{n,n_b} \clE_{n_b} + \sum_{k=n_b}^n \alpha^{n-k}_k \Xi_{n,k+1} \clE^\Delta_n
\, , \quad  
\Xi_{n,k} \eqdef  \prod^{n}_{i=k}  (1-\bdd{t:Lyapunov_eq_norm}\tfrac{1}{2}\alpha_{i+1})  
\end{aligned}
\]
Applying the bound $(1+\delta) \leq \exp(\delta)$ to  the above identity  with $\delta <\Xi $  gives, for a constant  $\bdde{t:stats_summary}$,	
\[
\clE_n
\leq 
\clE_{n_b} \exp( - \bdde{t:stats_summary} (\tau^b_n -\tau^b_{n_b}  )) + 
\sum_{k=n_b}^n  \alpha^{n-k}_k  \exp( - \bdde{t:stats_summary} (\tau^b_n - \tau^b_{k+1}))   \clE^\Delta_n
\]
in which 
$ \displaystyle
\sum_{k=1}^n \alpha_k 
\leq 
\tau^b_n \eqdef 
\alpha_0 (1+(1-\rho)^{-1}[n^{1-\rho} - 1])
$.

Since the noise is additive, \eqref{e:DeltaDecomp} yields $\Expect[\Delta_{k}] = \Expect[-\clT_{k+1} + \clT_{k}]$, in which $\clT_k$ is only a function of $\Phi_k$. The drift condition (DV3) implies that $\clE^\Delta_n \to 0$ geometrically fast (see \Cref{t:Vuni+CLT}~(i)), completing the proof.
\end{proof}


When $f$ is affine in $\theta$, the solution $\haf$ to Poisson's equation \eqref{e:fish-h} with forcing function $f$ takes the form $\haf(\theta,\Phi) = \haA(\Phi)\theta+\hab(\Phi)$. Consequently, $\Upupsilon_k$ is affine in $\theta$ for each $k$:
\begin{subequations}
\begin{align}
\Upupsilon_{k+1} &= (A_{k+1}-\haA_{k+1}) (A_{k+1} \theta_k + b_{k+1})
\label{e:upsilon_lin}
\\
\Upupsilon^*_{k+1} &= (A_{k+1}-\haA_{k+1}) (A_{k+1} \theta^* + b_{k+1})
\label{e:upsilon_lin_star}
\end{align}
\label{e:upsilon_lin_def}
\end{subequations}
We note that the definition for $\Upupsilon^*$ in \eqref{e:upsilon_lin_star} is agrees with the general version in \eqref{e:Upstar}. Moreover, the expression in \eqref{e:upsilon_lin_star} is the same as the major contributor for bias in linear SA with constant step-size, that is, $\barUpupsilon^*$ in \eqref{e:biasCDC} \cite{laumey23a}.

\begin{proof}[Proof of part (ii) of \Cref{t:stats_summary}] 
Taking expectations of both sides of \eqref{e:PRest_almost_as}, we obtain $\Expect[\Delta_{k+1}] = \Expect[-\clT_{k+1} + \clT_{k} - \alpha_{k+1} \Upupsilon_{k+1}]$ from \eqref{e:DeltaDecomp}, yielding
\[
A^* \Expect[\tilthetaPR_N] = \frac{1}{N}\Big( \Expect[S^\uptau_N] + \Expect[\clT_{N+1} - \clT_{1} ] + \sum_{k=1}^N  \alpha_{k+1}  \Expect[ \Upupsilon_{k+1}] \Big)
\]
The remainder of the proof consists of following the same steps as in the proof of \Cref{t:BigBounds}~(iv) to obtain the companion to \eqref{e:targetbias_rate}:
\begin{equation}
\Expect[\tilthetaPR_{N}]   =  \alpha_{N+1} \upbeta_\uptheta + \epsy_N^\upbeta 
\, ,  
\quad  
\| \epsy_N^\upbeta  \|    \le   \begin{cases}
	b^\circ N^{-3\rho/2} \quad &  \rho \leq 1/2
	\\
	b^\circ	N^{\rho/2-1} \quad &  \rho >1/2
\end{cases}
\label{e:SecondMomentBddLinearSA_bias}
\end{equation}	
with $  \upbeta_\uptheta = (1-\rho)^{-1} [A^*]^{-1} \barUpupsilon^* $.  Dividing both sides of the above equation by $\alpha_{N+1}$ and taking limits completes the proof for $\rho \in (0,2/3)$.
\end{proof}

\subsection{Asymptotic Covariance}
\label{s:Cov_proofs}
The next lemma will prove itself useful in establishing asymptotic covariances for PR averaging,

\begin{lemma}
\label[lemma]{t:Cov_lemma}
Suppose	that $\{\clV_N\}$ and $\{\clE_N\}$ are sequences of random variables such that $\| \clV_N\|_2^2 <\infty$ and $\| \clE_N\|_2^2<\infty$. Then, the following approximation holds
\[
\begin{aligned}
\Cov(\clV_N + \clE_N) &= \Cov(\clV_N) + \Sigma^\epsy_N
\\
\text{where} \quad
\Sigma^\epsy_N &\leq \epsy_N I 
\quad 
\text{with} \quad  
\epsy_N  = 2 \| \clV_N\|_{2,s} \| \clE_N \|_{2,s} + \|  \clE_N  \|_{2,s}^2
\end{aligned}
\]
\qed
\end{lemma}

\begin{proof}[Proof of part (iii) of \Cref{t:stats_summary} for the AD settting]
Adding and subtracting $\frac{1}{N}M^*_{N+1}$ to the right side of \eqref{e:PRest_almost_as} yields
\begin{subequations}
\begin{equation}
	\begin{aligned}
		\tilthetaPR_N &=\clV_N + \clE_N
		\\
		\text{in which }
		\clV_N &= [A^*]^{-1}\frac{1}{N} M^*_{N+1}  
		\\
		\clE_N &=  [A^*]^{-1}\frac{1}{N}  (S^\uptau_N   - S^\Delta_{N+1} - M^*_{N+1})
	\end{aligned}
	\label{e:PR_rep_preMSE}
\end{equation}
The expression \eqref{e:Sum_Delta_GenSA} gives $S^\Delta_{N+1}  = M_{N+1} - \clT_{N+1} + \clT_{1}$ for the additive noise setting. \Cref{t:Sum_delta_theta_bounds,t:Sum_delta_theta_bounds2} result in the following upper bounds: for a constant $b$ depending upon $\Psi_0$ and $\rho$, $\| \clV_N \|_2 \leq b  N^{-1/2}$ and
\begin{equation}
	\begin{aligned}
		\| \clE_N \|_2 
		&\leq 
		\| [A^*]^{-1} \|_F  \frac{1}{N}  (  \|  S^\uptau_N  \|_2 +  \| M_{N+1} - M^*_{N+1} \|_2 + \|\clT_{N+1} - \clT_{1} \|_2   )
		\\
		&\leq b  \max\{ N^{-(1+\rho)/2} , N^{-(2-\rho)/2}   \}
	\end{aligned}
\end{equation}
\nonumber
\end{subequations}
Consequently, the $L_2$ norm of $\clE_N $ is dominated by $\|S^\uptau_N  \|_2$ for $\rho >1/2$ and by $\| M_{N+1} - M^*_{N+1} \|_2$ when $\rho<1/2$. 

In view of these bounds and the upper bound after \eqref{e:L2andSpan}, we take covariances of both sides of \eqref{e:PR_rep_preMSE} and apply \Cref{t:Cov_lemma} to obtain
\[\begin{aligned}
\Cov(\tilthetaPR_N) &= \Cov(\clV_N + \clE_N) \leq  \Cov(\clV_N) + \epsy_N  I 
\, , \\  
\epsy_N &\leq  \bdd{t:stats_summary} \max\{ N^{-(2+\rho)/2},N^{-(3-\rho)/2}\}
\end{aligned}
\]
where $ \bdd{t:stats_summary}  = 3b^2$.
By definition, $M^*_{N+1} = \sum_{i=1}^N \MD^*_{k+1}$. Applying the martingale difference property, \eqref{e:AsymptCLT_phi} gives
\begin{equation}
\lim_{N \to \infty} N  \Cov(\tilthetaPR_N) 
=
\lim_{N \to \infty} N \Cov(\clV_N) 
= [A^*]^{-1}  \Sigma_{\MD^*} [{A^*}^\transpose]^{-1}   
\label{e:MD_AsympCov}
\end{equation}
\end{proof}

In view of \Cref{t:PRest_rep} and \eqref{e:Deltadecomp_recur_intro}, the following representation for PR-averaged estimates is obtained:
\begin{proposition}
\label[proposition]{t:PR_stoch_process}
Under the assumptions of \Cref{t:Met_Recursively}, the following holds:
\begin{equation}
A^* \tilthetaPR_{N} = -\frac{1}{N} S^{\clA}_N  - \frac{1}{N} H_{N+1}     - \frac{1}{N} J_{N+1}         +  \frac{1}{N} S^\circ_N 
\label{e:PR_stoch_process}
\end{equation}
where $\{S^\circ_N\}$ is a deterministic sequence, 
\begin{romannum}
\item $\{H_{N+1}\}$ is a martingale satisfying $\displaystyle \lim_{N \to \infty} \frac{1}{N} \Cov(H_{N+1}) = \Sigma_{\MD^*}$

\item
For a constant $\bdd{t:PR_stoch_process}$ depending upon $\Psi_0$ and $\rho$,  and $N\ge 1$,
\[ 
\| S^\clA_N \|_{2,s}
\leq 
\bdd{t:PR_stoch_process} \big[  1  + N^{1-\rho} \| \tilthetaPR_N \|_{2,s} +    \sum_{k = 1}^{N-1} k^{-\rho} \| \tilthetaPR_{k} \|_{2,s}  \big]
\]

\item
If in addition, $m = m(\rho)$ in  \eqref{e:Deltadecomp_recur_intro} is chosen so that $m\rho>1$, then for $N\ge 1$,
%
$ \displaystyle	 \| J_{N+1}   \|_2 \leq \bdd{t:PR_stoch_process} N^{\rho/2} $

\end{romannum}
\end{proposition}	

\begin{proof}
Summing \eqref{e:Deltadecomp_recur_intro} from $k=1$ to $N$ and substituting it into \eqref{e:PRest_almost_as} gives \eqref{e:PR_stoch_process} in which
\[
\begin{aligned}
S^{\clA}_N &= \sum_{k=1}^N \alpha_{k+1} \barG_{k+1} \tiltheta_k 
\, , \qquad \qquad   
S^\circ_N = \sum_{k=1}^N \alpha_{k+1}\upbeta^\circ_{k+1} 
\\
J_{N+1}  &=  - \clT^{\sbullet}_{N+1} +  \clT^{\sbullet}_{1} -S^\tau_N+ \sum_{k=1}^N [(\clR^\sbullet_{k}  - \clT^{\sbullet}_{k}) + (-\alpha_{k+1})^{m+1} \Upupsilon^{\sbullet}_{k+1}]
\end{aligned}
\]

We proceed with the proof of part (i). We have  $\MD^{(0)*}_{k+1} = \MD^*_{k+1}$ from \Cref{t:Met_Recursively},  
which along with  \eqref{e:Deltadecomp_recur_intro} and the definition of $\{H_{k}\}$ before \Cref{t:Markov_Deltarecur_bounds2} 
gives
the identity $\frac{1}{N} H_{N+1}  = \clV_N + \clE_N$, with $	\clV_N = \frac{1}{N} M^*_{N+1}$ and  $		\clE_N = \clE^1_N   + \clE^2_N$ where
\[
\begin{aligned}
\clE^1_N &=  \frac{1}{N} (H_{N+1} - H^{*}_{N+1})
\\
\clE^2_N &=
\frac{1}{N}  \sum_{k=1}^N \Big[ \sum_{i=1}^m \clW^{(i)*}_{k+1} 
\Big]
\end{aligned}
\]
Moment bounds on $\clV_N$ are obtained exactly as it was done in the proof of \Cref{t:stats_summary}~(iii) for the AD setting, while bounds on $\clE^1_N $ follow from \Cref{t:PR_stoch_process}~(i). They are of the form $\| \clV_N \|_2 \leq b N^{-1/2}$ and $\| \clE^1_N \|_2 \leq b \max\{N^{-(1+\rho)/2}, \sqrt{\log(N)}/N \}$, in which $b$ is a constant depending upon $\Psi_0$ and $\rho$.

Moreover,  similar arguments as the ones in \Cref{t:Markov_Deltarecur_bounds2}~(ii) and (iv) yield the following, for a potentially larger $b$,
\[
\begin{aligned}
\Big\| \sum_{k=1}^N \sum_{i=1}^m \clW^{(i)*}_{k+1}   \Big\|_2  & \leq b \max\{N^{(1-\rho)/2}, \sqrt{\log(N)} \}		
\end{aligned}
\]
where the first bound holds since the term associated with $i=1$ dominates. This implies $\|\clE^2_N  \|_2 \leq  3b \max\{N^{-(1+\rho)/2}, \sqrt{\log(N)}/N \}$.

Then, \Cref{t:Cov_lemma} and the upper bound after \eqref{e:L2andSpan} yield
\[\begin{aligned}
\Cov\Big( \frac{1}{N} H_{N+1} \Big) &=  \Cov\Big(\clV_N + \clE_N \Big)  \leq  \Cov(\clV_N) + \epsy_N  I 
\, , \\  
\epsy_N &\leq  \bdd{t:stats_summary} \max\{ N^{-1-\rho}, \log(N)/N^2\}
\end{aligned}
\]
where $\bdd{t:stats_summary} = 15b^2$.
Multiplying both sides of the above equation by $N$ and taking limits completes the proof for part (i):
\[
\lim_{N \to \infty} \frac{1}{N} \Cov(H_{N+1}) 
=
\lim_{N \to \infty} N \Cov(\clV_N) 
=  \Sigma_{\MD^*}
\]

We now turn to the proof of part (ii). 
Writing $\tiltheta_k = k\tilthetaPR_k - (k-1)\tilthetaPR_{k-1}$ gives
\[	\begin{aligned}
S^\clA_N
&=  \sum_{k=1}^N  \alpha_{k+1} \barG_{k+1} [k\tilthetaPR_k - (k-1)\tilthetaPR_{k-1}]
\\
&=    \alpha_{N+1} \barG_{N+1} N\tilthetaPR_N - \alpha_{2} \barG_{2} \tilthetaPR_1  - \sum_{k = 1}^{N-1} ( \alpha_{k+1} \barG_{k+1} - \alpha_{k} \barG_{k}) k\tilthetaPR_{k}
\end{aligned}
\]
where the last equality was obtained through \Cref{t:Sum_parts}.

Taking $L_2$ span norms of both sides of the above equation and using the triangle inequality 	
\[
\begin{aligned}
\|S^\clA_N\|_{2,s}  
&\leq 
|\alpha_{N+1} | \|  \barG_{N+1}  \|_F 
N \|    \tilthetaPR_N  \|_{2,s}  
+ |\alpha_{2} |\|  \barG_{2}  \|_F  \|   \tilthetaPR_1  \|_{2,s} 
\\
&+ \sum_{k=1}^{N-1} \|\alpha_{k+1}  \barG_{k+1}  -  \alpha_k  \barG_{k}\|_F  k \| \tilthetaPR_{k} \|_{2,s}  
\end{aligned}
\]  
This and     \Cref{t:Met_Recursively}~(iv) completes the proof of (ii).

Part (iii) follows from the definitions of $\{H_{N+1},J_{N+1}\}$ along with the bounds in \Cref{t:Met_Recursively} and \Cref{t:Sum_delta_theta_bounds,t:Hsolidarity}.

\end{proof}

The representation \eqref{e:PR_stoch_process} in \Cref{t:PR_stoch_process} is equivalent to \eqref{e:biasVariance}, in which $\clZ_n \eqdef -\sqrt{N}[A^*]^{-1}(H_{N+1} + J_{N+1}  +  S^\clA_N )$ and   $\upbeta_N \eqdef (N \alpha_{N+1})^{-1} S^\circ_N$.     

The representation  \eqref{e:PR_stoch_process} is particularly useful in the setting of linear SA since it results in tighter bounds for the $L_2$ span norm of $\thetaPR_N$ when compared to \Cref{t:BigBounds}~(iii). 
It will be clear that the asymptotic covariance of the right hand side of \eqref{e:PR_stoch_process} is dominated by $\{H_{N+1}\}$ and
equals $\Sigma_{\MD^*}$.  

Bounds on $\| \tilthetaPR_{N}\|_{2,s}$ are obtained next for the general nonlinear SA algorithm.


\begin{lemma}
\label[lemma]{t:PR_vanishingfast}
Suppose that the assumptions of \Cref{t:Met_Recursively} hold.
Then,
$
\| \tilthetaPR_N \|_{2,s}  \leq \bdd{t:PR_vanishingfast} N^{-1/2}
$,   $N\ge 1$, 
for a constant $\bdd{t:PR_vanishingfast}$ depending upon $\Psi_0$ and $\rho$.

\end{lemma}

\begin{proof}
Part (iii) of \Cref{t:BigBounds} and the inequality after \eqref{e:L2andSpan} give 
$\| \tilthetaPR_{N} \|_{2,s} \leq \bdd{t:BigBounds} \max\{\alpha_N,N^{-1/2}\}$, which completes the proof for $\rho\geq 1/2$.
We proceed with the proof of the desired bound  for $\rho <1/2$.

%
For each $\rho \in (0, 1/2)$, we choose $m$ in \eqref{e:Deltadecomp_recur_intro} so that $m \rho>1$.
Then, taking $L_2$ span norms of both sides of \eqref{e:PR_stoch_process} yields the 
recursive sequence of bounds, 
\begin{equation}
\begin{aligned}
\| A^* \tilthetaPR_{N} \|_{2,s} &\leq   \frac{1}{N}  \|S^\clA_N\|_{2,s}   + \frac{1}{N} \| H_{N+1} \|_{2,s}  + \frac{1}{N} \| J_{N+1} \|_{2,s} 
\\
& \leq \bdd{t:PR_stoch_process}  \frac{1}{N} 
\Big( N^{1-\rho} \| \tilthetaPR_N \|_{2,s}  
+  \sum_{k = 1}^{N-1} k^{-\rho} \| \tilthetaPR_{k} \|_{2,s} \Big) + \bdd{t:PR_stoch_process} N^{-1/2}
\end{aligned}
\label{e:forsub}
\end{equation}
where the last inequality follows from  \Cref{t:PR_stoch_process}, along with $\| S^\circ_N   \|_{2,s} =0  $.

Applying  \Cref{t:BigBounds}~(iii)  (the bound \eqref{e:PR_MSEBound})
gives  $ \| \tilthetaPR_{k} \|_{2,s}^2  \le \Expect[\| \thetaPR_k - \theta^* \|^2 ]  \leq \bdd{t:BigBounds}   \alpha_k^2 $  for $\rho<1/2$.
This bound together with \Cref{t:alpha_avg}~(ii) gives
\begin{equation}
\| A^* \tilthetaPR_{N} \|_{2,s} \leq \bdd{t:PR_stoch_process} \Big(\frac{1}{1-2\rho} + 1 \Big) \alpha_N^2 + \bdd{t:PR_stoch_process} N^{-1/2} \, , \quad \rho < 1/2
\label{e:boundrhohalf}
\end{equation}
completing the proof for $\rho \in [1/4,1/2)$.

If $\rho < 1/4$ then   \eqref{e:boundrhohalf}  gives $\| \tilthetaPR_k \|_{2,s} \leq \bdd{t:PR_vanishingfast} \alpha^2_k$, which is an improvement on  \Cref{t:BigBounds}~(iii).   Substituting this into \eqref{e:forsub} and repeating the above process yields a sequence of recursive bounds similar to \eqref{e:boundrhohalf}: 
\[
\| A^* \tilthetaPR_{N} \|_{2,s} \leq \bdd{t:PR_vanishingfast} \Big(\frac{1}{1-3\rho} +1 \Big) \alpha_N^3 + \bdd{t:PR_stoch_process} N^{-1/2} \, , \quad \rho < 1/4 
\]
which completes the proof for $\rho \in [1/6, 1/4)$. 

Continuing to repeat this process establishes the desired bound   for all $\rho \in (0,1)$.
\end{proof}

Now, we can apply the bounds in \Cref{t:PR_vanishingfast} to \Cref{t:PR_stoch_process}~(ii) 
and conclude that the sequences $\{S^\clA_N, J_{N+1}\}$ do not dominate the asymptotic covariance of the right side of \eqref{e:PR_stoch_process}.

\begin{proof}[Proof of part (iii) of \Cref{t:stats_summary} for the MU settting]   

It follows from \eqref{e:PR_stoch_process} that $\tilthetaPR_N  =\clV_N + \clE_N  $, in which
\begin{equation}
\clV_N = -[A^*]^{-1}\frac{1}{N} H_{N+1} \, , 
\qquad
\clE_N =  -[A^*]^{-1}\frac{1}{N}  (J_{N+1}  - S^\circ_{N}      + S^{\clA}_N    )
\end{equation}
Applying the bound $\| \tilthetaPR_N \|_{2,s} \leq \bdd{t:PR_vanishingfast} N^{-1/2}$ from \Cref{t:PR_vanishingfast} to \Cref{t:PR_stoch_process}~(ii), we obtain
\[
\|   S^{\clA}_N      \|_{2,s} 
\leq \bdd{t:PR_stoch_process} \bdd{t:PR_vanishingfast}  N^{(1-2\rho)/2}
\]
Applying \Cref{t:PR_stoch_process} gives,  for a constant $b$ depending upon $\Psi_0$ and $\rho$,
\[
\begin{aligned}
\|  \clV_N  \|_{2,s} & \leq   b N^{-1/2}
\\
\| \clE_N  \|_{2,s}
& \leq
\| [A^*]^{-1}\|_F \frac{1}{N}  \Big( \|  J_{N+1} \|_{2,s} + \|   S^{\clA}_N     \|_{2,s} \Big)
\\
& \leq b  \max\{   N^{-1+\rho/2}, N^{-1/2-\rho}  \}
\end{aligned}
\]

Applying \Cref{t:Cov_lemma} we obtain
\begin{equation}
\Cov(\tilthetaPR_{N})   =  \tfrac{1}{N}\SigmaPR + \clE_N^\Sigma 
\, ,   
\quad
\trace(
\clE_N^\Sigma  )  \le   \begin{cases}
\bdd{t:stats_summary}N^{-[1+\rho]} \quad &  \rho \leq 1/3
\\
\bdd{t:stats_summary}	N^{- [3 - \rho]/2} \quad&  \rho>1/3
\end{cases}
\label{e:SecondMomentBddLinearSA_cov}
\end{equation}
where $\bdd{t:stats_summary} = 3b^2$.
Multiplying both sides of the above equation by $N$ and taking limits completes the proof:  
%
\[
\lim_{N \to \infty} N  \Cov(\tilthetaPR_N) 
=
\lim_{N \to \infty} N \Cov(\clV_N) 
= [A^*]^{-1}  \Sigma_{\MD^*} [{A^*}^\transpose]^{-1}  
\]
\end{proof}

\begin{proof}[Proof of \Cref{t:stats_summary_better}]
Part (i) follows from taking norms of both sides of \eqref{e:SecondMomentBddLinearSA_bias} and applying the triangle inequality.

The proof of part (ii) begins by taking square roots of both sides of \eqref{e:L2andSpan} to obtain, via the triangle inequality
\[
\| \tilthetaPR_N \|_2  \leq  \sqrt{\trace(\Cov(\tilthetaPR_N))}  + \|   \Expect[ \tilthetaPR_N ]\|
\]
Substituting \Cref{e:SecondMomentBddLinearSA_cov,e:SecondMomentBddLinearSA_bias} into the above equation and using the triangle inequality once more completes the proof.
\end{proof}

\section{Numerical Experiments}
\label{s:exp_proofs}

Here we survey additional results and provide further details on the numerical experiments discussed in \Cref{s:exp}. 

Similarly to what is shown by  \Cref{fig:BiasVar}, \Cref{fig:BiasVar03,fig:BiasVar05} display the empirical mean and variance for the final PR-averaged estimates ${\thetaPR_{N}}^i$ across all independent runs  as functions of $\rho$ for $a \in \{0.3 ,0.5\}$. Additionally, ratios between these empirical statistics and their associated theoretical asymptotic predictions are also shown by \Cref{fig:All_a_comp} for each $a$. Note that there is no comparison with theoretical values for the empirical mean when $a=0.5$, in which case the Markov chain is i.i.d. and bias vanishes very fast, by \Cref{t:stats_summary}~(i).

The results in \Cref{fig:BiasVar03,fig:BiasVar05} agree with the discussion surrounding \Cref{fig:BiasVar}: poor solidarity with theory is observed in both large and small-$\rho$ regimes irrespective of the choice of $a$. Again, this is justified by the  error bounds in \Cref{t:stats_summary_better}.

\whamit{Impact of spectral gap.}   
The transition matrix $P$ has eigenvalues $\{\lambda_1, \lambda_2\} =  \{ 1,   2a-1\} $,  and the spectral gap is thus  $\gamma = 1 -|\lambda_2|  = 1 -| 2a-1| $.   
It is maximized when $a=0.5$ and vanishes if either $a\sim 0$ or $a\sim 1$.   
It is commonly assumed that a small spectral gap is a sign of difficulty in numerical algorithms, yet a glance at  \Cref{t:LinspecEx} shows that: 
\whamrm{(i)}   While it is true that $\barUpupsilon^* $  increases without bound when $a\uparrow 1$,   when $a\downarrow 0$ the spectral gap vanishes yet $\barUpupsilon^* $ converges to a finite value.
\whamrm{(ii)}   The asymptotic covariance \textit{vanishes} as $a\downarrow 0$, even as the spectral gap vanishes.
In conclusion:
long statistical memory is not always detrimental to algorithmic performance.  

The above conclusion is in agreement to what can be observed by the results in \Cref{fig:All_a_comp} since the empirical covariance for $a=0.3$ is closer to optimal than for the i.i.d. case ($a=0.5$). Moreover, the setting  $a=0.3$ also outperforms $a=0.7$ in terms of empirical variance even though these two choices result in the same spectral gap. 

When it comes to bias, it appears that $a=0.3$ and $a=0.5$ have similar performances: based upon \Cref{t:stats_summary}~(i), bias for the i.i.d. case vanishes very fast,  so the low magnitude observed for the empirical mean with $a=0.5$ is expected. The curve associated with $a=0.3$ seems to track the prediction in \Cref{t:stats_summary}~(ii) well only facing problems in the large $\rho$ regime. Tracking for case $a=0.7$ is not as good as when $a=0.3$ but also appears to be consisted of with theory. However, tracking is poor in both the low and large $\rho$ regimes.

\begin{figure}[h]
\centering
\includegraphics[width=1\textwidth]{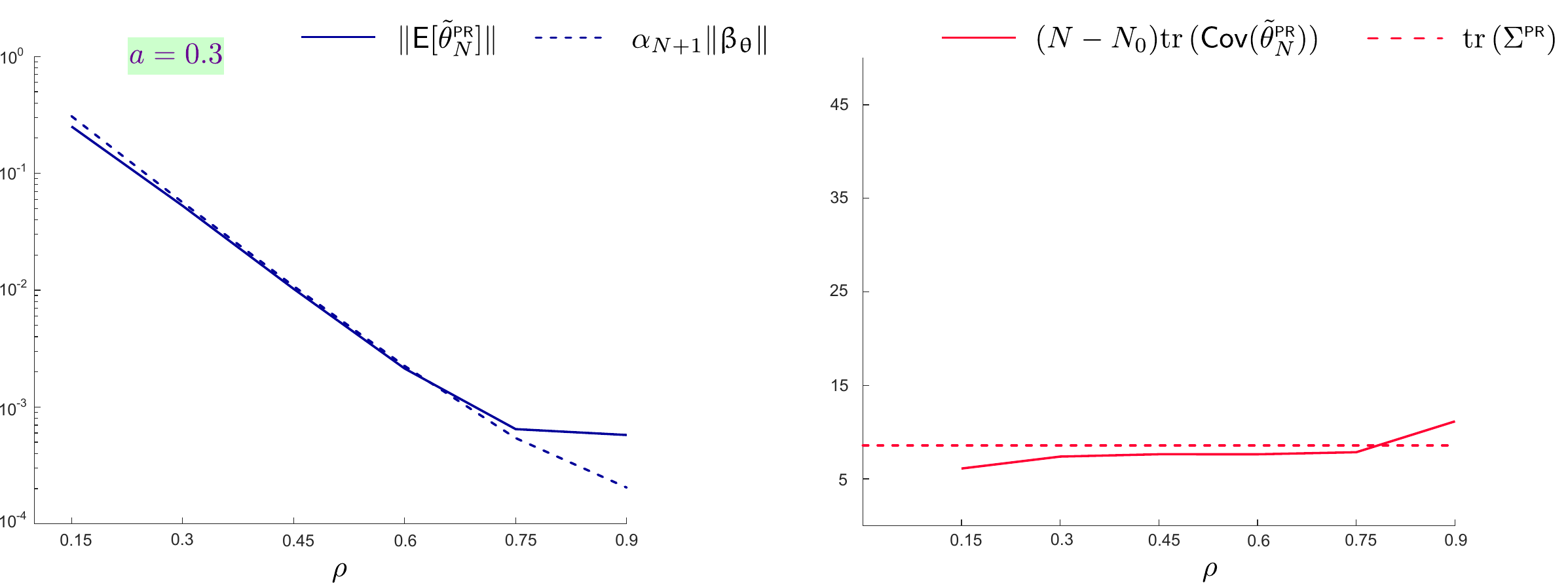}
\caption{Empirical and theoretical mean and variance with $a=0.3$.}
\label{fig:BiasVar03}
\end{figure}

\begin{figure}[h]
\centering
\includegraphics[width=1\textwidth]{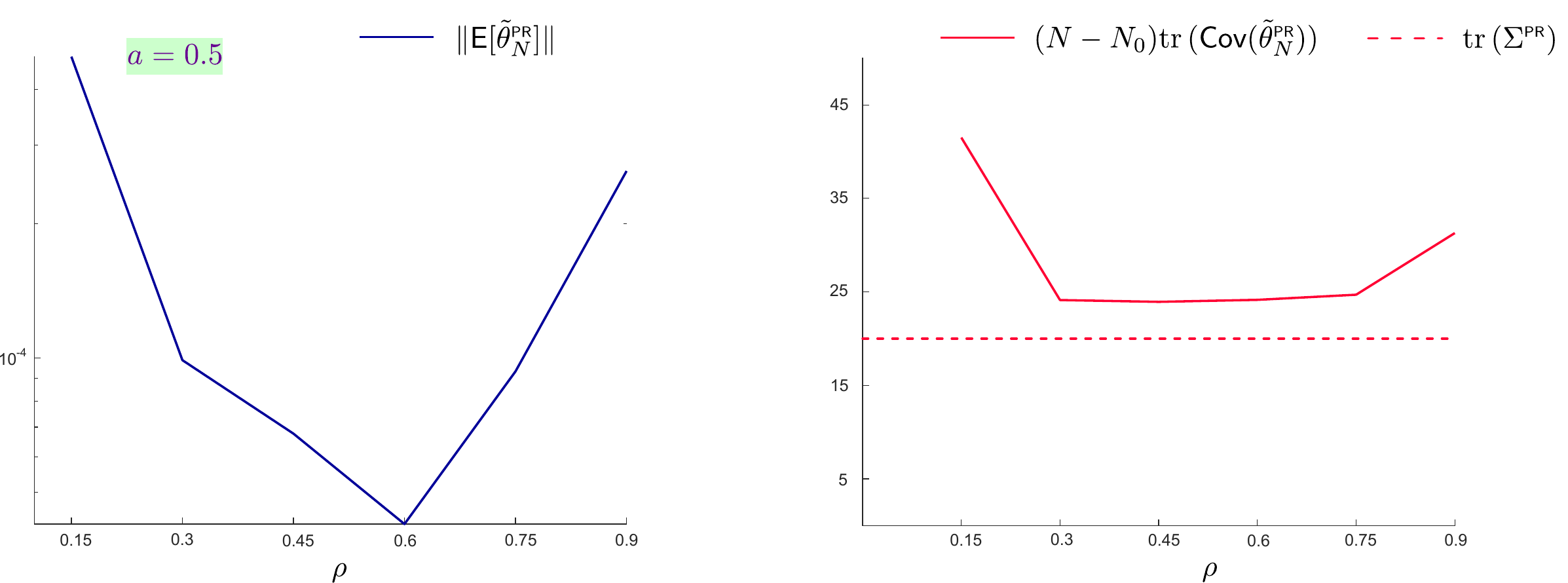}
\caption{Empirical and theoretical mean and variance with $a=0.5$.}
\label{fig:BiasVar05}
\end{figure}

\begin{figure}[h]
\centering
\includegraphics[width=1\textwidth]{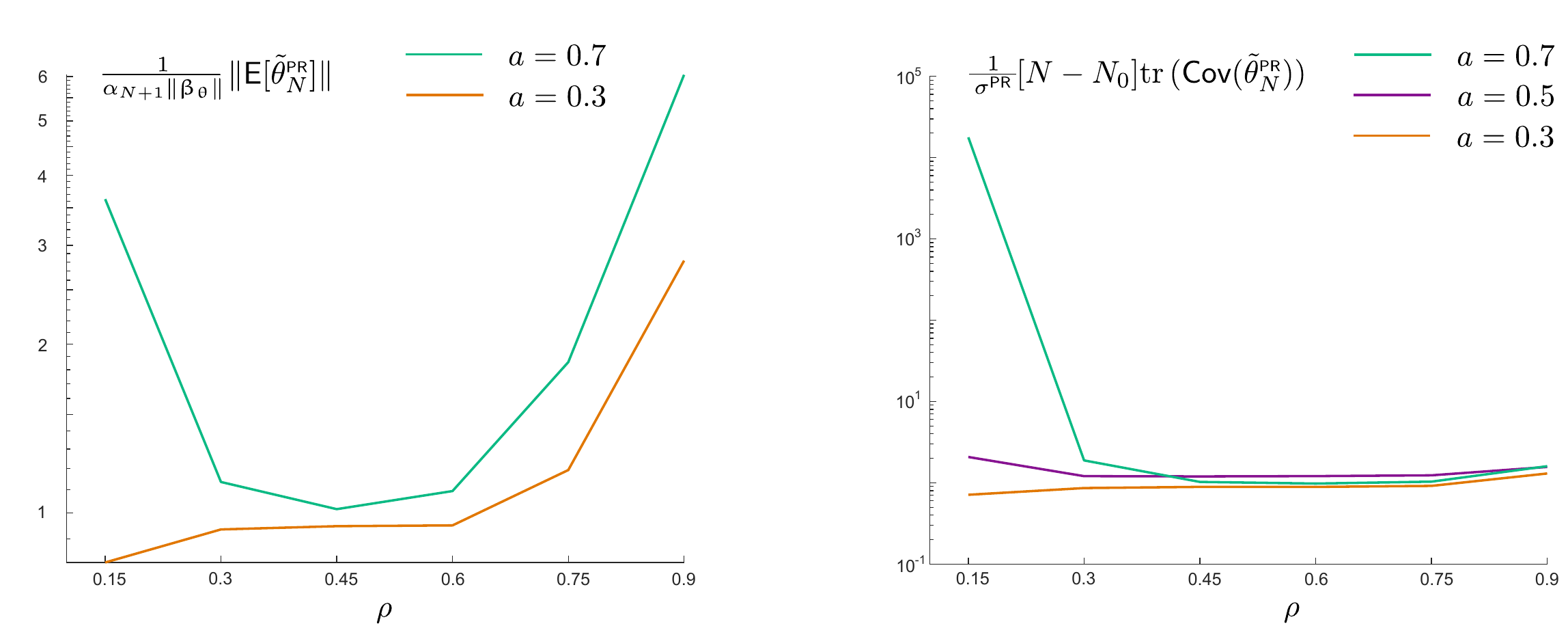}
\caption{Ratios between empirical and theoretical mean and variance for different choices of $a$.}
\label{fig:All_a_comp}
\end{figure}

\whamit{Identifying the constants in   \eqref{e:chezhadoaclamag22}.}  
The example considered in \Cref{s:exp} is of the form \eqref{e:model_SIVA} of  \cite{chezhadoaclamag22}, in which $X_k =\Phi_k$ and $ \MD^0_{k+1}\equiv 0$.  
This prior work  requires a quadratic Lyapunov function for the mean-flow vector field    (see discussion in \Cref{s:litreview}).
The values in \eqref{e:expchoices} were selected for the  model  \eqref{e:Affine_model_exp} so that $V(\theta) =\half \| \theta\|^2$  satisfies \eqref{e:Lyap_SIVA} with $c_0=1$.        The final term in  \eqref{e:chezhadoaclamag22} is thus
\[
L = 520 \alpha_0 (\|\theta^* \| + 1)^2 
\]

We conclude this section with a proof of \Cref{t:LinspecEx}, followed by details on the computing resources used to perform the experiments in \Cref{s:exp}.

\begin{proof}[Proof of \Cref{t:LinspecEx}]
Let $\hag:\state \to \Re$ denote the solution to Poisson's equation \eqref{e:bigfish} with forcing function $g(x)\equiv x$. Since $\uppi$ is uniform, Poisson's equation equation gives for each $x \in \state$,
\[
\Expect[\hag(\Phi_{k+1}) | \Phi_k =x] =  a \hag(x) + (1-a) \hag(\barx)   =\hag(x) -x + \tfrac{1}{2}
\]
where $\barx = \state \setminus \{x\}$. Solutions to \eqref{e:fish-h} are unique up to an additive constant, so let $\hag(1) =1$. Then, substituting $x=1$ into the above equation gives $\hag(0) = \tfrac{1}{2} \frac{1-2a}{1-a}$.
\begin{subequations}
It follows that $f(\theta,x)$ can be re-written as 
\begin{equation}
f(\theta,x) = [x(A^1 - A^0) + A^0] \theta - [x(b^1 - b^0) + b^0]
\end{equation}
which gives
\begin{align}
\haf(\theta,x) &= \hag(x)[(A^1 - A^0) \theta - (b^1 - b^0)] 
\\
\partial_\theta \haf(\theta,x) &= \hag(x)(A^1 - A^0)  
\end{align}
\label{e:def_linex}
\end{subequations}
where again we have used the fact that solutions to Poisson's equation are unique up to an additive constant.

The definition of $\Upupsilon^*$ in \eqref{e:Upstar} implies 
\[
\begin{aligned}
\Upupsilon^*_{k+1} &= ( [\Phi_{k+1}  -\hag(\Phi_{k+1})]  [A^1 - A^0] ) (A_{k+1} \theta^* - b_{k+1})
\\
\barUpupsilon^* = \Expect[\Upupsilon^*_{k+1}] &=  - \tfrac{1}{2}   (A^1 - A^0)\Big[ (A^0\theta^* - b^0)   \Expect[\hag(\Phi_{k+1}) |\Phi_{k} = 0 ]  +     (A^1\theta^* - b^1) (\Expect[\hag(\Phi_{k+1}) |\Phi_{k} = 1 ] -1)       \Big]
\\
& =  - \tfrac{1}{2} (A^1 - A^0)  \Big[ (A^0\theta^* - b^0)  \Big( \hag(0) -g(0)+\tfrac{1}{2}\Big)   +     (A^1\theta^* - b^1)  \Big( \hag(1) - g(1)+\tfrac{1}{2}\Big)       \Big]
\end{aligned}
\]
Part (i) is established upon using this with the following identity, obtained from \eqref{e:barf_linexp}:
\begin{equation}
0 = \barf(\theta^*) = \tfrac{1}{2} [(A^0+A^1) \theta^* - b^0-b^1 ] 
\label{e:barfstar_linex}
\end{equation}

For part (ii), it follows from \eqref{e:AsymptCLT_fish} that we can write the asymptotic covariance $\Sigma_\Delta$ as
\[
\Sigma_\Delta = \Expect_\uppi[  \Delta_k^*  {\haDelta_k^{*\transpose }}
+  \haDelta_k^*  {\Delta_k^*}^\transpose
-	\Delta_k^*  {\Delta_k^*}^\transpose	 ]
\]
where $\Delta_k^* = f(\theta^*,\Phi_{k})$ and $\haDelta_k^* = \haf(\theta^*,\Phi_{k})$. 
The law of total expectation and the definitions in \eqref{e:def_linex} give
\[
\begin{aligned}
\Expect_\uppi[\haDelta_k^*  {\Delta_k^*}^\transpose]
& =  \tfrac{1}{2} (\Expect_\uppi[\haDelta_k^*  {\Delta_k^*}^\transpose| \Phi_{k} = 0] +  \Expect_\uppi[\haDelta_k^*  {\Delta_k^*}^\transpose| \Phi_{k} = 1] )
\\
&= \tfrac{1}{2} [(A^1 - A^0) \theta^* - (b^1 - b^0)]  [ \hag(0)   (A^0\theta^* -b^0)^\transpose 
+\hag(1)    (A^1\theta^* -b^1)^\transpose ]
\\
\Expect_\uppi[	\Delta_k^*  {\Delta_k^*}^\transpose ] 
& = \tfrac{1}{2} (\Expect_\uppi[\Delta_k^*  {\Delta_k^*}^\transpose| \Phi_{k} = 0] +  \Expect_\uppi[\Delta_k^*  {\Delta_k^*}^\transpose| \Phi_{k} = 1] )
\\
&= \tfrac{1}{2} [(A^0\theta^* -b^0)(A^0\theta^* -b^0)^\transpose +  
(A^1\theta^* -b^1)(A^1\theta^* -b^1)^\transpose ] 
\end{aligned}
\]
Again, using \eqref{e:barfstar_linex} we obtain 
\[
\Expect_\uppi[\haDelta_k^*  {\Delta_k^*}^\transpose] = -(A^0\theta^*)(A^0\theta^*)^\transpose (\hag(0) - \hag(1))
\, , \quad 
\Expect_\uppi[	\Delta_k^*  {\Delta_k^*}^\transpose ]  = (A^0\theta^*)(A^0\theta^*)^\transpose
\] 
which completes the proof of (ii).
\end{proof}

\clearpage


\end{document}

%% file: bookmacros_2024.tex
\def\urls#1{{\footnotesize\url{#1}}}

 \def\baru{\bar{u}}
\def\SigmaCLT{\Sigma_{\text{\tiny \sf CLT}}}

\def\tilthetaPR{\tilde{\theta}^{\text{\tiny\sf  PR}}}

\def\thetaPR{\theta^{\text{\tiny\sf PR}}}

\def\sigmaPR{\sigma^{\text{\tiny\sf PR}}}

 \def\SigmaTheta{\Sigma_{\uptheta}}

\def\SigmaPR{\Sigma^{\text{\tiny\sf PR}}}

\def\whamb{\wham{$\bullet$}}

\def\bdd#1{b^{\text{\rm\tiny\ref{#1}}}}
\def\bdde#1{\varrho^{\text{\rm\tiny\ref{#1}}}}

\def\barM{\widebar{M}}

\def\MD{\clW}

\def\clT{\mathcal{T}}
\def\clD{\mathcal{D}}

\def\barG{\bar{G}}

\def\mindex#1{\index{#1}}

\def\ocp{*}   
  
 \def\baru{\bar{u}}

\DeclareFontFamily{U}{mathx}{\hyphenchar\font45}
\DeclareFontShape{U}{mathx}{m}{n}{<-> mathx10}{}
\DeclareSymbolFont{mathx}{U}{mathx}{m}{n}
\DeclareMathAccent{\widebar}{0}{mathx}{"73}

\def\barUpupsilon{\widebar{\Upupsilon}}

\def\odestate{\upvartheta}

\newcommand{\bbblot}{\raise1pt\hbox{\vrule height .4ex width .4ex depth .05ex}}

\long\def\defbox#1{\framebox[.9\hsize][c]{\parbox{.85\hsize}{%
\parindent=0pt
\baselineskip=12pt plus .1pt      
\parskip=6pt plus 1.5pt minus 1pt 
 #1}}}

\long\def\beginbox#1\endbox{\subsection*{}%
\hbox{\hspace{.05\hsize}\defbox{\medskip#1\bigskip}}%
\subsection*{}}

\def\endbox{}

 \def\archival#1{} 

\def\FRAC#1#2#3{\genfrac{}{}{}{#1}{#2}{#3}}

\def\ddt{{\mathchoice{\FRAC{1}{d}{dt}}%
{\FRAC{1}{d}{dt}}%
{\FRAC{3}{d}{dt}}%
{\FRAC{3}{d}{dt}}}}

\def\ddtp{{\mathchoice{\FRAC{1}{d^{\hbox to 2pt{\rm\tiny +\hss}}}{dt}}%
{\FRAC{1}{d^{\hbox to 2pt{\rm\tiny +\hss}}}{dt}}%
{\FRAC{3}{d^{\hbox to 2pt{\rm\tiny +\hss}}}{dt}}%
{\FRAC{3}{d^{\hbox to 2pt{\rm\tiny +\hss}}}{dt}}}}

\def\ddyp{{\mathchoice{\FRAC{1}{d^{\hbox to 2pt{\rm\tiny +\hss}}}{dy}}%
{\FRAC{1}{d^{\hbox to 2pt{\rm\tiny +\hss}}}{dy}}%
{\FRAC{3}{d^{\hbox to 2pt{\rm\tiny +\hss}}}{dy}}%
{\FRAC{3}{d^{\hbox to 2pt{\rm\tiny +\hss}}}{dy}}}}

\def\half{{\mathchoice{\FRAC{1}{1}{2}}%
{\FRAC{1}{1}{2}}%
{\FRAC{3}{1}{2}}%
{\FRAC{3}{1}{2}}}}

\def\state{{\sf X}}

\def\bx{{{\cal B}(\state)}}

\def\bx{{{\cal B}(\state)}}

\def\bfmath#1{{\mathchoice{\mbox{\boldmath$#1$}}%
{\mbox{\boldmath$#1$}}%
{\mbox{\boldmath$\scriptstyle#1$}}%
{\mbox{\boldmath$\scriptscriptstyle#1$}}}}

\def\bfPhi{\bfmath{\Phi}}
\def\bfPsi{\bfmath{\Psi}}

\def\bfDelta{\bfmath{\Delta}}

\def\bfmY{\bfmath{Y}}

\def\bfmhhaY{\bfmath{\hhaY}} 
\def\bfmhhaY{\hbox to 0pt{$\widehat{\bfmY}$\hss}\widehat{\phantom{\raise 1.25pt\hbox{$\bfmY$}}}}

\def\haf{{\hat f}}
\def\hag{{\hat g}}
\def\hah{{\hat h}}

\def\haA{\widehat A}

\def\tiltheta{{\tilde \theta}}

\def\tilg{\tilde g}

\def\clA{{\cal A}}
\def\clB{{\cal B}}

\def\clE{{\cal E}}
\def\clF{{\cal F}}

\def\clR{{\cal R}}

\def\clV{{\cal V}}
\def\clW{{\cal W}}

\def\clZ{{\cal Z}}

\def\eqdef{\mathbin{:=}}

\def\Prob{{\sf P}}

\def\Expect{{\sf E}}

\def\Cov{\hbox{\sf Cov}}

\def\lgmath#1{{\mathchoice{\mbox{\large #1}}%
{\mbox{\large #1}}%
{\mbox{\tiny #1}}%
{\mbox{\tiny #1}}}}

\def\Zero{{\mathchoice{\lgmath{\sf 0}}%
{\mbox{\sf 0}}%
{\mbox{\tiny \sf 0}}%
{\mbox{\tiny \sf 0}}}}

\def\as{{\rm a.s.}}

\def\ind{\bbbone}
  
 \def\epsy{\varepsilon}

\def\varble{\,\cdot\,}

\def\formtmp#1#2{{\vskip12pt\noindent\fboxsep=0pt\colorbox{#1}{\vbox{\vskip3pt\hbox to \textwidth{\hskip3pt\vbox{\raggedright\noindent\textbf{#2\vphantom{Qy}}}\hfill}\vspace*{3pt}}}\par\vskip2pt%
\noindent\kern0pt}}

\titleformat\subparagraph[runin]
                        {\normalfont\normalsize\bfseries}
                        {mypar}
                        {0pt}
                        {}{}
\titlespacing\subparagraph{0pt}
                       {.1ex minus 0.2ex}
                       {.75em}

\def\barb{{\overline {b}}}

\def\barf{{\widebar{f}}}
\def\barg{{\widebar{g}}}

\def\barh{{\overline {h}}}

\def\barx{{\overline {x}}}

\def\barA{{\bar{A}}}

\def\baralpha{{\bar{\alpha}}}

{\end{list}}

\def\ass(#1:#2){(#1\ref{#1:#2})}

\def\ritem#1{
\item[{\sf \ass(\current_model:#1)}]
}

\newenvironment{recall-ass}[1]{%
\begin{description}
\def\current_model{#1}}{
\end{description}
}

\def\sq{\hbox{\rlap{$\sqcap$}$\sqcup$}}
\def\qed{\ifmmode\sq\else{\unskip\nobreak\hfil
\penalty50\hskip1em\null\nobreak\hfil\sq
\parfillskip=0pt\finalhyphendemerits=0\endgraf}\fi}

\newcommand{\blot}{\vrule height 1.1ex width .9ex depth -.1ex }
\def\qedb{\ifmmode\blot\else{\vspace{-.2cm}\unskip\nobreak\hfil
\penalty50\hskip1em\null\nobreak\hfil\blot
\parfillskip=0pt\finalhyphendemerits=0\endgraf}\fi}

\newtheoremstyle{example}{15pt}{20pt}%
     {}
     {}
     {\bfseries}
     {}
     {1pt}
     {\thmname{#1}\thmnumber{ #2.}~\thmnote{\textit{\textbf{#3}}}%
     \\[.15cm]\unskip\nobreak}

\newcounter{rmnum}
\newenvironment{romannum}{\begin{list}{{\upshape (\roman{rmnum})}}{\usecounter{rmnum}
\setlength{\leftmargin}{8pt}
\setlength{\rightmargin}{8pt}
\setlength{\itemindent}{2pt}
}}{\end{list}}

\newcounter{anum}

\newcommand{\field}[1]{\mathbb{#1}}

\def\Re{\field{R}}

\def\Prob{{\sf P}}

\def\Expect{{\sf E}}

\def\transpose{{\intercal}}

\def\ind{\hbox{\large \bf 1}}

\def\trace{\hbox{\rm tr\,}}  

\def\epsy{\varepsilon}
\def\varble{\,\cdot\,}

\def\haY{\widehat{Y}}

\def\hhaY{\hbox to 0pt{$\haY$\hss}\widehat{\phantom{\raise 1.25pt\hbox{Y}}}}

\def\hab{\widehat b}

\def\haA{\widehat A}

\def\haG{{\widehat G}}

\def\haM{{\widehat M}}

\def\haY{\widehat Y}

\def\bfPhi{\bfmath{\Phi}}

\newlength{\dhatheight}

\def\barUpupsilon{\widebar{\Upupsilon}}
